\documentclass[11 pt, reqno]{amsart}
\usepackage{amsmath,amssymb,amsthm,amscd,amsfonts,mathrsfs}
\usepackage{latexsym}
\usepackage{xargs}
\usepackage{tikz}
\usepackage{graphicx,bbm}
\usepackage{stmaryrd}
\usepackage{pgf,tikz}
\usepackage{fancybox}
\usepackage{graphicx}
\usepackage{stmaryrd}
\usepackage{color}

\usepackage{algorithm}
\usepackage{algpseudocode}

\usepackage{mathpazo}
\usepackage[T1]{fontenc}
\usepackage{microtype}
\usepackage{soul}
\usepackage{multirow} 

\usepackage{placeins}
\usepackage{fullpage}
\newtheorem{thm}{Theorem}[section]

\newtheorem{lem}[thm]{Lemma}

\newtheorem{prop}[thm]{Proposition}

\newtheorem{cor}[thm]{Corollary}

\newtheorem{rmk}[thm]{Remark}
\newcommand{\be}{\begin{eqnarray}}
\newcommand{\ee}{\end{eqnarray}}
\newcommand{\ben}{\begin{eqnarray*}}
\newcommand{\een}{\end{eqnarray*}}
\newcommand{\beal}{\begin{aligned}}
\newcommand{\enal}{\end{aligned}}
\newcommand{\beq}{\begin{equation}}
\newcommand{\eeq}{\end{equation}}

\newcommand{\T}{\mathbb{T}}
\newcommand{\R}{\mathbb{R}}

\newcommand{\N}{\mathbb{N}}

\newcommand{\Z}{\mathbb{Z}}

\usepackage[unicode=true]{hyperref}
\hypersetup{
     colorlinks,
     linkcolor={black!10!blue},
     linkbordercolor = {black!100!blue},
     citecolor={red}
}

\usepackage[shortlabels]{enumitem}
\makeatletter
\def\namedlabel#1#2{\begingroup
    #2%
    \def\@currentlabel{#2}%
    \phantomsection\label{#1}\endgroup
}
\makeatother

\title{Operator Splitting, Policy Iteration, and Machine Learning for Stochastic Optimal Control}
\author{Alain Bensoussan$^*$, Thien P.B. Nguyen$^{**}$, Minh-Binh Tran$^{***}$, Son N.T. Tu$^\dagger$}

\address{$^{*}$Naveen Jindal School of Management, The University of Texas at Dallas, Richardson, TX 75080, USA}
\email{alain.bensoussan@utdallas.edu} 
\thanks{A.B.   is  funded in part by NSF Grant  DMS-2204795}

\address{$^{**}$Department of Mathematics, International University - Vietnam National University HCM City, Quarter 33, Linh Xuan Ward, Ho Chi Minh City, VN}
\email{bthien624@gmail.com}

\address{$^{***}$Department of Mathematics, Texas A\&M University, College Station, TX 77843, USA}
\email{minhbinh@tamu.edu} 
\thanks{M.-B. T is  funded in part by  a    NSF CAREER  DMS-2303146, and NSF Grants DMS-2204795, DMS-2305523,  DMS-2306379.}

\address{ $^\dagger$ Department of Mathematics, Baylor University, Waco, TX 76706, USA}
\email{Son\_Tu@baylor.edu}

\date{\today}

\numberwithin{equation}{section}

\newif\ifprintanswers
\printanswerstrue

\begin{document}

\begin{abstract}
We propose a splitting approach to solve the second-order Hamilton--Jacobi equation, reducing it to a heat step and a purely first-order step. The latter is implemented using a gradient value policy iteration algorithm, enabling efficient characteristic-based machine learning methods.
We establish convergence rates for the splitting method. In particular, with $h$ the splitting step, the $L^\infty$ error is bounded between $\mathcal{O}(h)$ and $\mathcal{O}(h^{1/5})$ for Lipschitz data, improving to $\mathcal{O}(h^{1/3})$ for semiconcave data. In the periodic setting, we also obtain an $L^1$ error of order $\mathcal{O}(h^{1/2})$.
For the first-order step, we provide a weighted $L^2$ error analysis that shows exponential convergence.
Each iteration solves linear characteristic equations and learns the value function by minimizing a weighted value gradient loss. The approach yields stable and accurate numerical results.
\end{abstract}

\maketitle

\section{Introduction}

Let $\varepsilon \in [0,1)$. We assume that $H:\R^d \times \R^d \to \R$ is a convex and coercive Hamiltonian, and that $u_0 \in W^{1,\infty}(\R^d)$. Let $u:\R^d \times [0,T] \to \R$ be the solution of
\begin{align}\label{eq:forwardtemp}
\begin{cases}
    \begin{aligned}
        u_t + H(x,Du) &= \varepsilon \Delta u &&\text{in}\; \R^d\times (0,T), \\
        u(x,0) &= u_0(x) &&\text{on}\; \R^d. 
    \end{aligned}
\end{cases}
\end{align}
The above Hamilton--Jacobi--Bellman (HJB) equation plays a central role in stochastic $(\varepsilon>0)$ and deterministic $(\varepsilon=0)$  optimal control 
\cite{Fleming1975Book,Bertsekas2001,fleming2006controlled,bensoussan2018estimation}.
Although the theoretical foundations of HJB equations are well established, their numerical solution remains highly challenging, both in the deterministic case ($\varepsilon=0$) and the stochastic case ($\varepsilon>0$). Closed-form solutions exist only in exceptional cases (e.g., linear-quadratic regulator problems, which reduce to Riccati equations \cite{bensoussan2018estimation}). In general, the PDE dimension matches the state dimension $d$, so traditional grid-based discretization methods suffer severely from the curse of dimensionality. This limitation has driven the development of modern techniques that integrate ideas from optimal control, reinforcement learning, and deep learning
\cite{sutton2018reinforcement,Bertsekas2019,Recht2019tour}.

Classical numerical strategies include successive approximation schemes
\cite{BEARD19972159,Bea1998,Beard:1998ti},
which iteratively linearizes the nonlinear HJB into a sequence of generalized HJ equations, each paired with pointwise minimization of the Hamiltonian. This procedure closely parallels policy iteration in reinforcement learning
\cite{sutton2018reinforcement,Bertsekas2019}.
For low-dimensional problems (small~$d$), spectral Galerkin methods remain viable \cite{Bea1998}. In moderate dimensions, polynomial expansions or low-rank tensor representations can be effective
\cite{7040310,KK2018,Oster2019}. For very high dimensions, deep neural networks have become the dominant approach.

When the Hamiltonian minimization step admits an explicit (or efficiently computable) solution, traditional grid-based schemes, including level-set methods~\cite{OSHER198812}, fast marching~\cite{FastMarching1994}, fast sweeping~\cite{FastSweep2003}, and semi-Lagrangian schemes~\cite{falcone2013semi}, provide accurate solutions, yet they still face exponential growth in computational cost and memory with respect to $d$.
To address truly high-dimensional regimes, many state-of-the-art neural PDE solvers leverage Lagrangian (or characteristics-based) reformulations of the HJB equation. Under suitable convexity or semiconcavity assumptions, generalized Lax--Hopf formulas recast value function computation as a convex optimization problem over terminal co-states along characteristic curves
\cite{DarbonOsher2016RMS,ChowOsher2017JSC,ChowOshersplitting2018,ChowOsherADMM2018,ChowOsher2019JSC}. An alternative route relies on Pontryagin's maximum principle, which leads to a two-point boundary value problem in the combined state--costate space
\cite{Kang2015causality}. Solutions to this system can then be interpolated using sparse grids~\cite{Kang2017} or trained neural networks in a supervised manner
\cite{Izzo01145,Kang2019,Kang2021QRnet}; see also recent high-dimensional PDE survey \cite{2020HighPDEReivewE} and related works on using Machine Learning and Policy iterations for Hamilton--Jacobi--Bellman equations
\cite{lee_hamiltonjacobi_2025, yang_solving_2025,kim_neural_2025,kim_physics-informed_2025}.
\medskip

We propose an algorithm to solve \eqref{eq:forwardtemp} based on an operator splitting scheme. In this context, operator splitting means separating the nonlinear effect in \eqref{eq:forwardtemp}. 
Operator splitting methods for nonlinear partial differential equations have been studied in the literature. In \cite{souganidis_max-min_1985}, such methods were analyzed for purely first-order equations, with convergence rates provided, see also \cite{jakobsen_convergence_2001,langseth_convergence_1996}. An abstract convergence theory for general schemes was developed in \cite{barles_convergence_1991}. Later, \cite{talay_convergence_1997} studied a splitting method for nonlinear degenerate elliptic and parabolic equations arising in option pricing models. See also the related method of dimensional splitting in \cite{sun_alternating_1996,luca_gerardo-giorda_parallelizing_2014}.

\subsection{The Splitting Algorithm}
In this approach, we split the evolution into a pure heat step and a purely first-order Hamilton–Jacobi step. The heat step is computationally straightforward to implement, while the first-order step is solved using a value–gradient (also known as PI-$\lambda$) policy iteration method, first introduced in \cite{2020handbookXZ} and later studied for the time-independent first-order case in \cite{bensoussan_value-gradient_2023}.
   This first-order step has the advantage of being easily 
   implemented in parallel and exhibits exponential convergence of the 
   policy iteration in a weighted $L^2$ norm. Moreover, 
   it can be implemented using a machine learning algorithm with fast 
   convergence \cite{bensoussan_value-gradient_2023,2020handbookXZ}. 
   Each iteration solves linear equations along characteristics from 
   sampled initial states. Training minimizes a convex combination 
   of mean-squared errors for the value function and its gradient. 
   The value function is approximated using nonparametric models 
   (e.g., radial basis functions or neural networks), 
   with gradients obtained via automatic differentiation.
   Our numerical experiments confirm the accuracy and stability 
   of the proposed scheme. 
   We also provide a rigorous error analysis of the algorithm.

    First, we provide an error analysis of the splitting scheme. 
    The error with respect to the splitting time step improves as 
    the regularity of the initial data $u_0$ increases; 
    see Subsection \ref{subsection:main-results} below for the main results.
    In a sense, the splitting scheme \eqref{eq:split-scheme} can be viewed 
    as a Trotter--Kato product, which is known to converge qualitatively 
    in related settings.
    Quantitative error estimates for such schemes remain largely unexplored; 
    Theorem~\ref{thm:LinftyA} provides, to our knowledge, 
    the first result of this type.

    Second, we show that the Hamilton–Jacobi step exhibits exponential 
    convergence in a weighted $L^2$ norm. 
    This analysis extends the results for the first-order 
    time-independent problem established in 
    \cite{bensoussan_value-gradient_2023}. 
    Since we consider a time-dependent problem, 
    we have the flexibility to modify the weighted error in a natural way. 
    This contrasts with the analysis in \cite{bensoussan_value-gradient_2023}, 
    where the result holds only for sufficiently large values 
    of the discount factor.
    One advantage of this splitting scheme is 
    that the first-order step is purely deterministic, 
    which makes both backward and forward time computations 
    straightforward, as well as the characteristic method 
    implemented via machine learning (see Section~\ref{sec:ML}).

\subsubsection{The Splitting Scheme}

For clarity, we introduce the solution operators as follows. 
Let $u(\cdot,t)=\mathcal{S}^{\mathrm{H}}_t u_0$ denote the solution of the heat equation
    \begin{align}\label{eq:heat} 
    \begin{cases} 
    \begin{aligned} 
        u_t -\varepsilon\Delta u&=0, && (x,t)\in\R^d\times(0,T),\\ 
        u(x,0)&=u_0(x), && x\in\R^d.
    \end{aligned} 
    \end{cases} 
    \end{align}
    A formula for $\mathcal{S}^{\mathrm{H}}$ is provided by the heat kernel: for $t>0$ and $u_0\in L^2(\R^d)$, set
    $$
    \mathcal{S}^{\mathrm{H}}_t u_0(x)
        = \int_{\R^d}\Phi_\varepsilon(x-y,t)u_0(y)\,dy, 
    $$
    where $\Phi_\varepsilon(x,t)
        = (4\pi\varepsilon t)^{-d/2}
        e^{-\frac{|x|^2}{4\varepsilon t}}$.  
We consider the first-order Hamilton--Jacobi equation:
    \begin{align}\label{eq:HJperator}
    \begin{cases}
        \begin{aligned}
            u_t + H(x, Du) &= 0 &&\text{in}\;\R^d\times (0,T), \\
            u(x,0) &= u_0(x) &&\text{on}\;\R^d.
        \end{aligned}
    \end{cases}
    \end{align}
    Under standard assumptions, \eqref{eq:HJperator} admits a unique viscosity solution $u(\cdot,t)$, 
    and we define $u(\cdot,t)=\mathcal{S}^{\mathrm{HJ}}_t u_0$ 
    for $t\geq 0$ 
    (see \cite{bardi_optimal_1997,crandall_users_1992,fathi_weak_2008,tran_hamilton-jacobi_2021}).  
  Given $n\in\mathbb{N}$, we set $h = T/n$ and $t_i = ih$ for $i=0,\dots,n$. Our result takes the form
\begin{align}\label{eq:split-scheme}
    u(x,ih) \approx \left( \mathcal{S}^{\mathrm{H}}_{h} \circ \mathcal{S}^{\mathrm{HJ}}_{h}\right)^i u_0(x), \qquad (x,t)\in \R^d\times (0,T),
\end{align}
where $u$ is the true solution to \eqref{eq:forwardtemp}.
The splitting approximation $v$ is defined iteratively by
\begin{align}\label{eq:v-split-intro}
\begin{cases}
\begin{aligned}
    v(x,t_i) &= \mathcal{S}^{\mathrm{H}}_{h} \circ \mathcal{S}^{\mathrm{HJ}}_{h} v(\cdot, t_{i-1})(x) = \mathcal{S}^{\mathrm{H}}_{h} \zeta^{(i)}(\cdot, h)(x), \qquad x\in \R^d, \\
    v(x,t_0) &= u_0(x).
\end{aligned}
\end{cases}\\
\text{where}\qquad 
\begin{cases}
    \begin{aligned}
        \partial_t\zeta^{(i)}(x,t) + H\big(x,D\zeta^{(i)}(x,t)\big) &= 0 &&\quad \text{in}\;\R^d\times (0,h), \\
        \zeta^{(i)}(x,0) &= v(x,t_{i-1}) &&\quad \text{on}\;\R^d. 
    \end{aligned}
\end{cases}
\end{align}

\subsubsection{Policy Iteration for first-order Hamilton--Jacobi equation} \label{subsubsection:HJB1st}
For the first-order problem \eqref{eq:HJperator}, we assume that the Hamiltonian $H$ arises from an optimal control problem (see \eqref{eq:DefnHInf}), so that the value function admits a convenient variational formulation.

    For $t\in [0,T]$, we denote by $\mathcal{A}_t = \{a(\cdot):[0,t]\to \R^p \mid a(\cdot)\ \text{is measurable}\}$ the set of admissible controls. Elements of $\mathcal{A}_t$ are referred to as controls. When talking about feedback control, we mean $a(t) = a(\gamma(t))$ where $a: \R^d\to \R^p$ and $\gamma:[0,t]\to \R^d$ is a state-trajectory.
    We refer to $a:\R^d\to \R^p$ as a policy. 

    The dynamics $f:\R^d\times \R^p \to \R^d$ is assumed to be globally Lipschitz continuous. The running cost is $\ell:\R^d\times \R^p \to \R$, and the initial cost is 
$u_0:\R^d \to \R$.

For $a(\cdot) \in \mathcal{A}_t$ and $(x,t)\in \R^d\times [0,T]$, we denote by 
$\eta(\cdot)\in \mathrm{AC}([0,t];\R^d)$ the solution to
\begin{align}\label{eq:eta}
    \begin{cases}
    \begin{aligned}
        \dot{\eta}(s) &= -f(\eta(s), a(s)), \qquad s\in (0,t), \\ 
            \eta(t)&=x.
    \end{aligned}
    \end{cases}
\end{align}
Then, we denote by $u:\R^d\times [0,T]$ the value function of the minimization problem
\begin{align}\label{eq:valueuHJB}
    u(x,t) = \inf_{a(\cdot)\in \mathcal{A}_t} 
    \left\lbrace 
        \int_0^t \ell(\eta(s), a(s))ds + u_0(\eta(0)): \eta(t)=x, \eta\;\text{solving}\; \eqref{eq:eta} 
    \right\rbrace. 
\end{align}
It is well known, see \cite{tran_hamilton-jacobi_2021, bardi_optimal_1997, cannarsa_semiconcave_2004}, that $u$ satisfies the Dynamic Programming Principle and is a viscosity solution of \eqref{eq:HJperator}, where the associated Hamiltonian $H:\R^d\times\R^d\to\R$ is defined by
\begin{align}\label{eq:DefnHInf}
    H(x,p) = \sup_{a \in \R^p} \Big( p\cdot \big(-f(x,a)\big) - \ell(x,a) \Big), 
    \qquad (x,p) \in \R^d \times \R^d .
\end{align}
Policy iteration is an effective algorithm for solving first-order Hamilton--Jacobi equations. We refer to the recent developments in \cite{alla_efficient_2015, tang_policy_2025, guo_policy_2025} and the references therein. 
To solve for $u(x,t)$, we seek the optimal policy using the following algorithm. Let $\lambda(x,t)=Du(x,t)$. Then $\lambda$ satisfies
\begin{align}\label{eq:PDEforLambda}
\begin{cases}
\begin{aligned}
    & \lambda_t = D^T\lambda\cdot f(x,\hat{a}(x,t)) + D ^Tf(x,\hat{a}(x,t))
                \cdot \lambda + \nabla_x \ell(x,\hat{a}(x)), \\
    & \lambda(\cdot ,0) = \nabla u_0 .
\end{aligned}
\end{cases}
\end{align}
As noted in \cite{bensoussan_control_2019,2020handbookXZ}, the Algorithm \ref{alg:PI-lambda-time} decouples the components of $\lambda^{(k+1)}(x)$, allowing them to be solved in parallel. Each of the $d$ equations has the same form, so the method of characteristics applies (see Section \ref{sec:ML}).

\FloatBarrier
\begin{algorithm}[H]
\caption{PI-$\lambda$ for first-order time-dependent problem}\label{alg:PI-lambda-time}
\begin{enumerate}
    \item[Step 1.] Given $a^{(k)}(x,t)$, we solve for $\lambda^{(k+1)}(x,t)$:
    \begin{align}\label{eq:lambda-k}
    \begin{cases}
    \begin{aligned}
        \partial_t \lambda^{(k+1)}(x,t) &= D\lambda^{(k+1)}\cdot f(x,a^{(k)}) + D _xf(x,a^{(k)})\cdot \lambda^{(k)} + \nabla_x \ell(x,a^{(k)}), \\
        \lambda^{(k+1)}(x,0) &= \nabla u_0(x) .
    \end{aligned}
    \end{cases}
    \end{align}
    
    \item[Step 2.] $a^{(k+1)}$ is obtained from the optimization problem 
    \begin{align}\label{eq:akmin}
    	a^{(k+1)}(x,t) = \arg\max_a 
		    \Big( 
        		\lambda^{(k+1)}(x,t)\cdot (-f(x,t)) -  \ell(x,a)
		    \Big). 
	\end{align}
\end{enumerate}
\end{algorithm}

\subsection{Assumptions And Main Results}\label{subsection:main-results}

The first result concerns the error estimate for the splitting scheme \eqref{eq:forwardtemp} 
in terms of the step size $h=T/n$ as $n\to\infty$, with $n$ the number of splitting steps.
We state assumptions \ref{itm:H1}, \ref{itm:H2} 
used to prove the error estimate for the splitting scheme.

\begin{itemize}


    \item[\namedlabel{itm:H1}{$(\mathbf{H}_1)$}] 
    Let \(H \in C^2(\R^d \times \R^d)\) be such that 
    \(H, D_pH \in \mathrm{BUC}(\R^d \times \overline{B}_R(0))\) 
    for each \(R>0\), and
      $$
        \lim_{|\xi|\to \infty} 
        \inf_{x\in \R^d} 
        \left(\frac{1}{2}H(x,\xi)^2 + D_xH(x,\xi)\cdot \xi \right) = +\infty .
     $$
    Moreover, for each \(R>0\) there exists a constant \(C_R>0\) such that for all
    \(x,y,p,q \in \R^d\) with \(|p|,|q|\le R\),
    \begin{align}\label{eq:continuousH-H1}
        \begin{cases}
            |H(x,p) - H(y,p)| \le C_R |x-y|,\\[2pt]
            |H(x,p) - H(x,q)| \le C_R |p-q|.
        \end{cases}
    \end{align}

    \item[\namedlabel{itm:H2}{$(\mathbf{H}_2)$}] $H\in C^2(\R^d\times \R^d)$ such that for all $(x,\xi)\in \R^d\times \R^d$ then
    \begin{align*}
        \begin{cases}
        \begin{aligned}
            & H(x,\xi)\;\text{is uniformly convex in}\;\xi: && \theta \ \mathbb{I}_d\preceq D^2_{pp}H(x,\xi)  \\
            & \text{Linear Growth First Derivative:} && |D_{x}H(x,\xi)| \leq \Lambda_H(1+|\xi|) \\ 
            & \text{Bounded Second Derivatives:} &&  |D^2_{xx}H(x,\xi)|, |D^2_{xp}H(x,\xi)| \leq \Lambda_H.
        \end{aligned}
        \end{cases}
    \end{align*}
    We denote
    $C_H = \sqrt{\Lambda_H\left(\frac{2}{\theta} + \frac{4}{\theta^2}\right)}$.    

\end{itemize}

\begin{thm}\label{thm:LinftyA} Assume \ref{itm:H1}--\ref{itm:H2}. Let $u_0 \in W^{1,\infty}(\R^d)$. Let $u$ and $v$ denote the true solution of \eqref{eq:forwardtemp} and the splitting scheme \eqref{eq:v-split-intro}, respectively, with initial data $u_0$.
\begin{itemize}
    \item[$(\mathrm{i})$] For all $x\in \R^d$ and $i=1,\ldots,n$
    \begin{align*}
        -CT \varepsilon h \leq u(x,t_i) - v(x,t_i)
    \end{align*}
    where $C$ is a constant depending only on $H$ and $\mathrm{Lip}(u_0)$. 


    \item[$(\mathrm{ii})$] For all $x\in \R^d$ and $i=1,\ldots,n$,
    \begin{align*}
        u(x,t_i) - v(x,t_i) \leq 
        \begin{cases}
            \begin{aligned}
                &C(T\varepsilon h)^{1/5}    &&\quad \;u_0\;\text{is Lipschitz}\\    
                &C(T\varepsilon h)^{1/3}    &&\quad \;u_0\;\text{is Lipschitz and semiconcave}
            \end{aligned}
        \end{cases}
    \end{align*}
     where $C$ depends only on $H$, $\mathrm{Lip}(u_0)$, and, in the latter case, also on the semiconcavity constant $\mathrm{SC}(u_0)$. 

\end{itemize}
\end{thm}

The upper convergence rate improves to $\mathcal{O}(\sqrt{\varepsilon h})$ in the 
$L^1$-norm in the periodic setting, using a BV estimate for gradients of semiconcave functions.

\begin{prop}\label{prop:errorTorusTL} 
Assume \ref{itm:H1}--\ref{itm:H2}, $x \mapsto H(x,p)$ is $\mathbb{Z}^d$-periodic, 
and $u_0\in W^{1,\infty}(\R^d)$ be semiconcave. 
%
Let $u$ and $v$ denote the true solution of \eqref{eq:forwardtemp} and the splitting scheme \eqref{eq:v-split-intro}, respectively, with initial data $u_0$. For $i=1,\ldots, n$ we have
\begin{align}\label{eq:estiTdThL}
  -C_1T\cdot \varepsilon h \leq     
  \int_{\T^d}
    \big( u(x, t_{i}) - v(x, t_i) \big) 
    dx 
    \leq 
        C_2 (e^{C_2T}-1)\cdot (\varepsilon h)^{1/2} . 
\end{align}
Here $C_1$ depends only on $H$ and $\mathrm{Lip}(u_0)$, while $C_2$ also depends on $d$ and $\mathrm{SC}(u_0)$.
\end{prop}

\begin{rmk}
    We note that a similar result in $L^1_\omega(\R^d)$,
    where $\omega$ is either a smooth cutoff function or a weight 
    decaying sufficiently fast at infinity, can also be 
    obtained along the lines of \eqref{eq:estiTdThL}, 
    with the same upper rate $\mathcal{O}(\sqrt{\varepsilon h})$. 
    The proof would rely on the ingredients used in the commutator 
    estimate of Proposition \ref{prop:commutatorL1} below. 
    We omit the details here. 
\end{rmk}

The second main result concerns the first-order step and the convergence of Algorithm~\ref{alg:PI-lambda-time}. Using the value–gradient method \cite{2020handbookXZ}, we derive an error estimate in a weighted space–time $L^2$ norm. The analysis is more involved than in \cite{bensoussan_value-gradient_2023}, which treats the time-independent case. We assume additional optimal-control structure on the Hamiltonian, enabling the $L^2_\omega$ error analysis and the machine learning implementation, and allowing unbounded initial data at the cost of stronger assumptions on the dynamics and cost.\medskip

\begin{itemize}
    \item[\namedlabel{itm:A0}{$(\mathbf{A}_0)$}] 
    $|\nabla u_0(x)| \leq C_0(1+|x|)$ and $|D^2u_0(x)|\leq C_0$ for all $x\in \R^d$. 
    
    \item[\namedlabel{itm:A1}{$(\mathbf{A}_1)$}] The dynamic $f(x,a):\R^d\times \R^p \to \R$ is given by
    \begin{align}
        & f(x,a) =   f_0(x) + \mathbf{c}_f^T\cdot a,  
        \quad f_0:\R^d\to \R^d, 
        \quad \mathbf{c}_f    \in \mathbb{M}^{p\times d}, |\mathbf{c}_f|\leq C_f
        \label{eq:formf}\\
        & |f(x,a)| \leq C_f (1+|x|+|a|),  
        \quad  |D_xf(x,a)| \leq C_f. 
        \label{eq:growthf}
    \end{align}
    \noindent 
    The running cost $\ell(x,a): \R^d \times \R^p \to \R$ satisfies
    \begin{align}
		&D^2_{aa} \ell (x,a) \succeq \ell_0 \;\mathbb{I}_d 
            \quad\;(x,a)\in \R^d\times \R^p\label{eq:strongconvexell}, 
        \quad |D_a \ell(x,a)| \geq \ell_0(|a| - 1) \\ 
		&|\ell(x,a)| \leq C_\ell (1+|x|^2+|a|^2) , 
        \quad  |D_x\ell(x,a)| \leq C_\ell(1+|x|+|a|). 
    \end{align}

    \item[\namedlabel{itm:A2}{$(\mathbf{A}_2)$}] 
     The dynamic $f(x,a):\R^d\times \R^p \to \R$ satisfies further that
    \begin{align}
         & \sum_{i=1}^d \left|D_x\big(\partial_{x_i}f(x,a)\big)\right| \leq \frac{C_f}{1+|x|} . 
        \label{eq:2ndDerivativeDf}
    \end{align}
    The running cost $\ell(x,a): \R^d \times \R^p \to \R$ satisfies further that
    \begin{align*}
        & |\ell(x_1,a_1) - \ell(x_2,a_2)| \\
        & \qquad \leq C_\ell(1+\max(|x_1|, |x_2|) + \max(|a_1|,|a_2|))\big(|x_1-x_2|+|a_1-a_2|\big)
        \\ 
        &  
        |D^2_{xa} \ell(x,a)|, |D^2_{ax} \ell(x,a)|, |D^2_{xx} \ell(x,a)| \leq C_\ell.   
    \end{align*}
    \color{black}
\end{itemize}

\begin{thm}\label{thm:mainthm2}
    Assume \ref{itm:A0}--\ref{itm:A2}. Let $\alpha \in (0,1)$ and $\gamma>0$. 
    For $k=1,2,\ldots$, define the $L^2$-weighted error
    \begin{align}\label{eq:error}
        e_k = \int_0^T e^{-\gamma t}\int_{\R^d} \frac{|(\lambda^{(k)}- \lambda^{(k-1)})(x,t)|^2}{(1+|x|^2)^{2\alpha}}\;dxdt. 
    \end{align}
    There exists $T_0>0$ and $\gamma_0>0$, depending on \ref{itm:A0}--\ref{itm:A2}, 
    such that for $T\leq T_0$ and $\gamma\geq T_0$, 
    $e_k=\mathcal{O}(2^{-k})$ as $k\to\infty$. 
    More precisely, $e_k\leq 2^{-(k-1)}e_1$ for all $k\in\N$.
\end{thm}

\subsection{Approach}
To analyze the error of the splitting scheme 
\eqref{eq:split-scheme} between the true solution
\(u\) and \(v\) from \eqref{eq:v-split-intro}, 
we introduce an appropriate comparison function, 
following 
\cite{jakobsen_convergence_2001,langseth_convergence_1996}. 
\medskip

For the lower bound of $u-v$, we take the comparison 
function directly from \eqref{eq:v-split-intro}; 
on each interval 
$$
    v(\cdot, t_i + t) = \mathcal{S}^{\mathrm{H}}_t\zeta^{(i)}(\cdot, t)    
$$
The lower bound for $u-v$ reduces to a commutator estimate between $\mathcal{S}^{\mathrm{H}}_t$ and $H$. 
By exploiting the convexity of $H$, we obtain a linear error of order $\mathcal{O}(h)$.
The argument starts from Lipschitz data and relies on the stability of the Lipschitz norm along the iteration, which follows from properties of first-order Hamilton--Jacobi equations and the nonexpansiveness of the heat operator.
\medskip 

For the upper bound of \(u-v\), the approach above using 
the comparison function \(\zeta^{(i)}\) for \(v\) leads, 
through the commutator, to terms of the form 
\(D\zeta^{(i)}(x-y,t)-D\zeta^{(i)}(x,t)\), 
which require higher regularity of \(D\zeta^{(i)}\).
When \(\zeta^{(i)}\) is semiconcave, one can instead exploit BV estimate for $D\zeta^{(i)}$ (see Subsections \ref{Subsection:Notations} and \ref{Supp:BVfunctions} of the Supplementary Material). 
Starting from semiconcave data yields \(D\zeta^{(i)} \in \mathrm{BV}_{\mathrm{loc}}(\R^d)\), 
which provides an \(L^1\)-weighted estimate and uniform control near \(t=0\). 
This result is stated in Proposition~\ref{prop:errorTorusTL}.
To obtain the $L^\infty$ estimate, we introduce a new comparison function that is more regular than the 
splitting approximation \eqref{eq:split-scheme}, 
possessing $C^2$ regularity. This is achieved by 
using the vanishing viscosity approximation of the 
first-order equation with parameter $\delta$. 
Precisely, we define
$$
        v_\delta(\cdot, t_i+t) = \mathcal{S}^{\mathrm{H}}_t \zeta^{(i)}_\delta(\cdot, t)
$$
where $v_\delta(\cdot, 0) = u_0$, and iteratively,
\begin{align}\label{eq:vanishing}
\begin{cases}
    \begin{aligned}
        \partial_t\zeta^{(i)}_\delta + H\big(x,D\zeta^{(i)}_\delta\big) &= \delta \Delta \zeta_\delta^{(i)} &&\text{in}\;\R^d\times (0,h), \\
        \zeta_\delta^{(i)}(x,0) &= v_\delta(x,t_{i-1}) &&\text{on}\;\R^d. 
    \end{aligned}
\end{cases}
\end{align}
We then proceed to estimate $ u - v = (u - v_\delta) + (v_\delta - v)$.
The $C^2$-regularity of $\zeta^{(i)}_\delta$ allows us to establish the commutator estimate, thereby obtaining the upper bound for $u - v_\delta$. The remaining error $v_\delta - v$ reduces to the classical vanishing viscosity error, for which semiconcavity yields the optimal upper rate $\mathcal{O}(\delta t)$; see below. We note that simply regularizing $\zeta^{(i)}$ by convolution is insufficient, as the resulting error is $\mathcal{O}(\delta)$ rather than $\mathcal{O}(\delta t)$, which is needed to obtain a valid global error estimate.

The main difficulty is to control the error at each step so that the local error is better than linear. This requires stability of the Lipschitz and semiconcavity bounds, which are preserved under the heat operator. The most delicate part is obtaining a uniform lower bound for $D^2\zeta^{(i)}$, since the heat operator accumulates errors at each iteration, worse than the Hamilton--Jacobi step alone. This estimate relies crucially on the exact linear lower bound of the commutator in the lower bound; any weaker rate would not suffice for our approach.

Since the vanishing viscosity error \eqref{eq:vanishing} appears in our proof, we briefly review the literature on this well-studied problem and list a non-exhaustive selection of related works below.
The rate $\mathcal{O}(\delta^{1/2})$ has been established in many works \cite{bardi_optimal_1997,crandall_two_1984,fleming_convergence_1961,lions_generalized_1982,qian_optimal_2024,tran_hamilton-jacobi_2021}, and was also obtained in \cite{evans_adjoint_2010,tran_adjoint_2011} via the nonlinear adjoint method. A rate $\mathcal{O}(\delta \log(\delta))$ in one dimension was proved in \cite{qian_optimal_2024}, and similar results have been obtained in higher dimensions for uniformly convex Hamiltonians in \cite{chaintron_optimal_2025,cirant_convergence_2025,wang_vanishing_2025}. Moreover, an averaged $L^1$ rate of order $\mathcal{O}(\varepsilon)$ was proved in \cite{tran_hamilton-jacobi_2021}, and linear $L^p$ rates in the semiconcave setting were obtained in \cite{camilli_quantitative_2024}. It is also well known that the upper bound for $w^\delta - w$ improves to $\mathcal{O}(\delta)$ under suitable semiconcavity assumptions \cite{bardi_optimal_1997,lions_generalized_1982,calder2024} (see \cite{han_remarks_2022,dutta_rate_2026} for the case of state constraints). We remark that, for our purposes, the optimal upper bound 
$v_\delta - v = \mathcal{O}(\delta)$ obtained using semiconcavity estimates is all we need.

\subsection{Notations} \label{Subsection:Notations}
For \(R>0\), let \(B_R\) and \(\overline{B}_R\) denote the open and closed balls in \(\mathbb{R}^d\) centered at \(0\) with radius \(R\), respectively.
A function $u:\R^d\to \R$ is said to be semiconcave with constant $\Lambda$ if $D^2u(x)\preceq \Lambda \ \mathbb{I}_d$ for all $x\in \R^d$ in the distribution sense; equivalently, $u(x) - \frac{\Lambda}{2}|x|^2$ is concave. We denote  
\begin{align}\label{eq:SC}
    \mathrm{SC}(u):= \inf \left\lbrace \Lambda\in \R: u(x)-\tfrac{1}{2}\Lambda|x|^2\;\text{is concave}\right\rbrace.
\end{align}
If $u$ is Lipschitz, namely $|u(x)-u(y)|\leq C|x-y|$ for all $x,y\in \R^d$ we denote by $\mathrm{Lip}(u)$ its Lipschitz constant. The space of Lipschitz function
s is denoted by $\mathrm{Lip}(\R^d)$. We write \(\|u\|_\infty = \|u\|_{L^\infty(\R^d)}\) when convenient.
We denote by $\Vert D^2u\Vert_\infty$ the $L^\infty$ norm of the Hessian of $u$ when it exists, namely $|\langle D^2u(x) \xi, \xi\rangle| \leq \Vert D^2u\Vert_{\infty}$ for all $x\in \R^d$ and $\xi\in \R^d$ with $|\xi|=1$.

Let $\Omega \subset \mathbb{R}^d$ be open. A function $u \in L^1(\Omega)$ belongs to $BV(\Omega)$ if its distributional gradient $Du$ is a finite $\mathbb{R}^d$-valued Radon measure. The total variation of $Du$ is denoted by $|Du|$, and $u \in BV(\Omega)$ if and only if $|Du|(\Omega) < \infty$ (see \cite{cannarsa_semiconcave_2004, evans_measure_2015})). 

We denote by $\int_{\R^d}|x| \Phi_\varepsilon(x,t)dx = \mathbf{C}_d(\varepsilon t)^{1/2}$, 
where $\Phi_\varepsilon$ is the heat kernel. 
The constant $\mathbf{C}_d$ depends only on the dimension $d$ 
and is given by $\mathbf{C}_d = \frac{2}{\pi^{d/2}}\int_{\R^d} |x|e^{-|x|^2}dx$. We also denote by $u^+ = \max\{u,0\}$.

\subsection{Organization of the paper}
In Section~\ref{sec:basic-est}, we first collect some basic estimates for the heat equation and the Hamilton--Jacobi equation, mainly for second-order equations, as well as the key commutator estimate. In particular, we track the Lipschitz constant, the semiconcavity constant, and a lower bound on the Hessian of the solution throughout the iteration obtained by concatenating the Hamilton--Jacobi and heat steps. We present the $L^\infty$ error analysis of the splitting scheme, proving Theorem~\ref{thm:LinftyA}, and the weighted $L^1$ error analysis, proving Proposition \ref{prop:errorTorusTL}.
In Section~\ref{sec:FirstOrder}, we establish the estimates for the first-order algorithm~\ref{alg:PI-lambda-time} and provide a proof of Theorem~\ref{thm:mainthm2}. 
Section~\ref{sec:ML} is devoted to describing the machine learning method used to compute the solution component of Algorithm~\ref{alg:PI-lambda-time}, and we present several numerical tests demonstrating the effectiveness of our approach.
Proofs of several technical results are deferred to the Supplemental Material.


\section{Error Analysis of the Splitting Scheme} \label{sec:basic-est}
\subsection{Basic properties and commutator estimates in $L^\infty(\R^d)$} 
In this subsection, we collect several useful results on the Heat and Lax–Oleinik operators, along with estimates needed for the subsequent analysis. First, we collect basic estimates for the heat operator $\mathcal{S}^{\mathrm{H}}$, which follow immediately from the heat kernel representation.

\begin{lem}\label{lem:heatestimates} 
    The heat operator $\mathcal{S}^{\mathrm{H}}_t: L^\infty(\R^d)\to L^\infty(\R^d)$ 
    is monotone, i.e., $\mathcal{S}^{\mathrm{H}}_t f \leq \mathcal{S}^{\mathrm{H}}_t g$
    for all $f,g\in L^\infty(\R^d)$ with $f\leq g$ in $\R^d$.
    Let $v(x,t) = \mathcal{S}^{\mathrm{H}}_{t}v_0(x)$ for $(x,t)\in \R^d\times (0,\infty)$. We have 
    \begin{align}
    &\Vert v(\cdot,t) - v_0(\cdot)\Vert_{L^\infty(\R^d)} \leq C_d\Vert Dv_0\Vert_{L^\infty(\R^d)}(\varepsilon t)^{1/2}
             && v_0\in \mathrm{Lip}(\R^d)\label{eq:heatInitialRateLipschitz} \\
    &\Vert v(\cdot,t) - v_0(\cdot)\Vert_{L^\infty(\R^d)} \leq 
        \Vert \Delta v_0\Vert_{L^\infty(\R^d)}
        \varepsilon t 
                && v_0\in C^2(\R^d). 
        \label{eq:heatInitialRateC2} 
\end{align}
Furthermore, for every $f \in W^{1,\infty}(\R^d)$ and every $t \geq 0$, the spatial derivative commutes with the semigroup:
\begin{align}\label{eq:changeSD}
    D(\mathcal{S}^{\mathrm{H}}_t f)(x) = \mathcal{S}^{\mathrm{H}}_t (Df) (x) \qquad\text{for every}\;x\in \R^d. 
\end{align}
\end{lem}

Next, we state basic results for \(\mathcal{S}^{\mathrm{HJ}}\); see \cite{tran_hamilton-jacobi_2021} for proofs.

\begin{prop}\label{prop:1stHJB}
Assume \ref{itm:H1}. 
    \begin{itemize}
        \item[$\mathrm{(i)}$] \emph{(Comparison Principle)}
        If $u$ is a continuous viscosity subsolution, and $v$ is a continuous viscosity supersolution to \eqref{eq:HJperator}, then $u\leq v$. 
        
        \item[$\mathrm{(ii)}$] 
        For each $u_0\in W^{1,\infty}(\R^d)$ there exists a unique viscosity solution to \eqref{eq:HJperator} with $u(\cdot, t)\in W^{1,\infty}(\R^d)$ for every $t\in  [0,T]$. Moreover, 
        \begin{align*}
            \left| u(x,t) - u_0(x) \right| & \leq t\left( \sup \lbrace
            |H(x, Du_0(x))|: x \in \R^d \rbrace   \right) , 
            \\
             \Vert Du(\cdot,t)\Vert_{L^\infty(\R^d)} &\leq C \qquad \text{where}\qquad C = C(H, \mathrm{Lip}(u_0)). 
        \end{align*}
    \end{itemize}
\end{prop}

The first-order problem \eqref{eq:HJperator} is well posed under weaker
assumptions on $H$ (see \cite{tran_hamilton-jacobi_2021}).  
We impose the stronger condition \ref{itm:H1} for 
technical convenience in the analysis of the 
second-order case. We note that an explicit 
expression for the Lipschitz bound 
of $\mathrm{Lip}(u(\cdot,t))$ is available when 
$\xi \mapsto H(x,\xi)$ is coercive with a clear 
structure (see \cite{armstrong_viscosity_2015}).


Next, we derive commutator estimates between the 
heat operator \(\mathcal{S}^{\mathrm{H}}_t\) and 
the Hamiltonian \(H\), which are central to 
the error analysis of the splitting scheme. 
We refer to \cite{le_dynamical_2017} for a related 
estimate vanishing discount problem (in the periodic setting).

\begin{prop}[Commutator Estimates in $L^\infty$] 
    \label{prop:CommutatorEstimates}
Assume \ref{itm:H1}--\ref{itm:H2}, and let $t>0$. 
If $\zeta \in \mathrm{Lip}(\mathbb{R}^d)$, there exists $\Lambda_1(\zeta)$ depending only on $H$ and $\mathrm{Lip}(\zeta)$ such that
    \begin{align}\label{eq:CommutatorLower}
         \mathcal{S}^{\mathrm{H}}_t \left(x,D\zeta(x)\right) - H\big(x,\mathcal{S}^{\mathrm{H}}_t D\zeta(x)\big) 
         \geq 
         -\Lambda_1(\zeta)\cdot \varepsilon t.
    \end{align}
If $\zeta \in \mathrm{Lip}(\mathbb{R}^d)$ and $\sup_{|v|\leq 1} |\langle D^2\zeta(x) v, v\rangle| \leq \Vert D^2\zeta\Vert_{\infty}$ 
for all $x\in \R^d$, then
\begin{align}\label{eq:CommutatorUpper}
  \mathcal{S}^{\mathrm{H}}_t \left(x,D\zeta(x)\right) - H(x,\mathcal{S}^{\mathrm{H}}_t D\zeta(x))  
        \leq 
        \Lambda_2(\zeta)\cdot d\varepsilon t, 
\end{align}
where 
\begin{align}
    \Lambda_2(\zeta)
        &= 3\Lambda_H \Vert D^2\zeta\Vert_\infty + \Lambda_3(\zeta)\Vert D^2\zeta\Vert_\infty^2 +\Lambda_H  \label{eq:Lambda2Zeta}\\ 
    \Lambda_3(\zeta) 
        &=  \sup \{|D^2_{pp}H(x,\xi)|: x\in \R^d, |\xi|\leq \mathrm{Lip}(\zeta)\}.  \label{eq:Lambda3Zeta}
\end{align}
\end{prop}

\ifprintanswers


\begin{proof}[Proof of Proposition \ref{prop:CommutatorEstimates}] 
Under \ref{itm:H1}--\ref{itm:H2}, the Lagrangian of $H$, given by the Legendre transform $L(x,v) = \sup_{p\in \R^d} \left(p\cdot v - H(x,p)\right)$ for $(x,v)\in\R^d \times \R^d$, satisfies that $L\in C^2(\R^d\times \R^d)$. 

For the lower bound, let $C_R = \sup \{|D^2_{xx}L(x,v)|: (x,v)\in \R^d\times \overline{B}_R\}$. 
There exists $\Gamma_\zeta>0$ such that the Hamiltonian satisfies (see \cite[p.~67]{tran_hamilton-jacobi_2021})
\begin{align}\label{eq:Hrestricted}
    H(x,p) = \sup_{|v|\leq \Gamma_\zeta} \big(p\cdot v - L(x,v)\big), \qquad |p|\leq \mathrm{Lip}(\zeta). 
\end{align}
By Lemma \ref{lem:heatestimates}-\eqref{eq:heatInitialRateC2} we have 
$
\sup_{|v|\leq \Gamma_\zeta}\Vert \mathcal{S}^{\mathrm{H}}_t L(\cdot,v) - L(\cdot,v)\Vert_{L^\infty(\R^d)} \leq \Lambda_1(\zeta)\varepsilon t,   
$
where $\Lambda_1(\zeta) = \sup \left\lbrace |\Delta_{xx}L(x,v)|: x\in \R^d,|v|\leq \Gamma_\zeta \right\rbrace$. 
For any $v\in \R^d$ with $|v|\leq \Gamma_\zeta$, we have
$$
\begin{aligned}
    \mathcal{S}^{\mathrm{H}}_t \big(v\cdot D\zeta(x)-L(x,v)\big) - \big(v\cdot \mathcal{S}^{\mathrm{H}}_t D\zeta(x) - L(x,v)\big) 
    \geq -\Lambda_1(\zeta)\cdot  \varepsilon t 
\end{aligned}
$$
thanks to Lemma \ref{lem:heatestimates}-\eqref{eq:changeSD}. Therefore
\begin{align*}
    \mathcal{S}^{\mathrm{H}}_t \big(v\cdot D\zeta(x)-L(x,v)\big) \geq 
    \big(v\cdot \mathcal{S}^{\mathrm{H}}_t D\zeta(x) - L(x,v)\big) -\Lambda_1(\zeta) \cdot \varepsilon t, \qquad |v|\leq \Gamma_\zeta. 
\end{align*}
By Lemma \ref{lem:heatestimates}, $\mathcal{S}^{\mathrm{H}}_t$ is monotone. Hence, since $ H(x,D\zeta(x)) \geq v\cdot D\zeta(x)-L(x,v)$ for any $|v|\leq \Gamma_\zeta$, we deduce that 
\begin{align*} 
    \mathcal{S}^{\mathrm{H}}_t H(x,D\zeta(x)) 
    \geq 
     \mathcal{S}^{\mathrm{H}}_t \big(v\cdot D\zeta(x)-L(x,v)\big)\geq 
    \big(v\cdot \mathcal{S}^{\mathrm{H}}_t D\zeta(x) - L(x,v)\big) - \Lambda_1(\zeta)\cdot \varepsilon t.  
\end{align*}
Taking the supremum over $|v| \leq \Gamma_\zeta$ and using \eqref{eq:Hrestricted}, we obtain
\begin{equation*}
    \mathcal{S}^{\mathrm{H}}_t H(x,D\zeta(x))  
    \geq 
    \sup_{|v|\leq \Gamma_\zeta} \big(v\cdot \mathcal{S}^{\mathrm{H}}_t D\zeta(x) - L(x,v)\big) = 
    H\left(x, \mathcal{S}^{\mathrm{H}}_t D\zeta(x)\right) - \Lambda_1(\zeta)\cdot \varepsilon t.  
\end{equation*}


For the upper bound, we have
\begin{align}\label{eq:comm}
    & \mathcal{S}^{\mathrm{H}}_t \left(x,D\zeta(x)\right) - H(x,\mathcal{S}^{\mathrm{H}}_t D\zeta(x)) 
    = A+B,
\end{align}
where
\begin{align*}
    A &= \int_{\R^d} 
    \Big( 
        H(x-y, D\zeta(x-y)) - H(x-y, \mathcal{S}^{\mathrm{H}}_tD\zeta(x))
    \Big) \Phi_\varepsilon(y,t)dy \\ 
    B &= \int_{\R^d} 
    \Big( 
        H(x-y, \mathcal{S}^{\mathrm{H}}_tD\zeta(x)) - H(x, \mathcal{S}^{\mathrm{H}}_tD\zeta(x))
    \Big) \Phi_\varepsilon(y,t)dy. 
\end{align*}

\noindent 
{\it Step 1. Estimate for $A$.} 
By Taylor's expansion we have $A=A_1+A_2$, where 
\begin{align}
    A_1 &=  \int_{\R^d} D_pH\big(x-y, \mathcal{S}^{\mathrm{H}}_t D\zeta(x)\big)\cdot \Big(D\zeta(x-y)-\mathcal{S}^{\mathrm{H}}_t D\zeta(x)\Big)\Phi_\varepsilon(y,t)dy \label{eq:A1}\\
    A_2 &= \int_{\R^d} \left(\int_0^1 (p-p_x)\cdot D^2_{pp}H(z, p_x)\cdot (p-p_x) (1-s)\;ds \right)  \Phi_\varepsilon(y,t) dy, \label{eq:A2}
\end{align}
where $p = D\zeta(x-y)$ and $p_x = \mathcal{S}^{\mathrm{H}}_t D\zeta(x)$. 
For $A_1$, using
\begin{align*}
    \int_{\R^d} D_pH(x, \mathcal{S}^{\mathrm{H}}_t D\zeta(x))\cdot (D\zeta(x-y)-\mathcal{S}^{\mathrm{H}}_t D\zeta(x))\Phi_\varepsilon(y,t)dy = 0
\end{align*}
and the Fundamental Theorem of Calculus 
we obtain
\begin{align*}
    & A_1 
    = \int_{\R^d} 
        \int_0^1 D_{xp}H(x-sy,\mathcal{S}^{\mathrm{H}}_t D\zeta(x))\;ds 
        \cdot(-y)\cdot \left(D\zeta(x-y)-\mathcal{S}^{\mathrm{H}}_t D\zeta(x)\right)\Phi_\varepsilon(y,t)dy. 
\end{align*}
Therefore, by \ref{itm:H2} we have
\begin{align*}
    |A_1| &\leq 
        \Lambda_H \int_{\R^d} \int_{\R^d} |y|
    \left|
        \int_{\R^d} |D\zeta(x-y) - D\zeta(x-z)|
    \right|\Phi_\varepsilon(z,t)\Phi_\varepsilon(y,t)dzdy \\
    &\leq 
        \Lambda_H\Vert D^2\zeta\Vert_{\infty} \int_{\R^d}\int_{\R^d} |y||y-z| \Phi_\varepsilon(z,t)\Phi_\varepsilon(y,t)dzdy \\
    & \leq 
        \Lambda_H\Vert D^2\zeta\Vert_{\infty} \left(\int_{\R^d} |y|^2\Phi_\varepsilon(y,t)dy \right)^{1/2}
    \left(\int_{\R^d}\int_{\R^d} |y-z| ^2\Phi_\varepsilon(z,t)\Phi_\varepsilon(y,t)dzdy \right)^{1/2} \\
    &= 
        \Lambda_H\Vert D^2\zeta\Vert_{\infty} \cdot (2d \varepsilon t) ^{1/2} \cdot (4d \varepsilon t) ^{1/2}  
        \leq 
        \Lambda_H\Vert D^2\zeta\Vert_{\infty} (3d)\cdot \varepsilon t. 
 \end{align*}
For $A_2$, using $\Lambda_3(\zeta) =  \sup \{|D^2_{pp}H(x,\xi)|: x\in \R^d, |\xi|\leq \mathrm{Lip}(\zeta)\}$ we obtain
\begin{align*}
    \int_0^1 (p-p_x)\cdot D^2_{pp}H(z, p_x)\cdot (p-p_x) (1-s)\;ds \leq \frac{\Lambda_3(\zeta)}{2} |p-p_x|^2. 
\end{align*}
Therefore
\begin{align*} 
    A_2 
    &\leq 
        \frac{\Lambda_3(\zeta)}{2}\int_{\R^d} \left| D\zeta(x-y) - \mathcal{S}^{\mathrm{H}}_t D\zeta(x) \right|^2\Phi_\varepsilon(y,t)dy \\
    &= 
        \frac{\Lambda_3(\zeta)}{2} \int_{\R^d} \left|\int_{\R^d} \Big(D\zeta(x-y) - D\zeta(x-z)\Big) \Phi_\varepsilon(z,t)dz\right|^2\Phi_\varepsilon(y,t)dy\\
    &\leq 
        \frac{\Lambda_3(\zeta)}{2} \int_{\R^d}\int_{\R^d} |D\zeta(x-y)-D\zeta(x-z)|^2 \Phi_\varepsilon(z,t)\Phi_\varepsilon(y,t)dzdy \\
    &\leq 
        \frac{\Lambda_3(\zeta)}{2} \int_{\R^d} \int_{\R^d} \Vert D^2\zeta\Vert_{\infty}^2 |y-z|^2 \Phi_\varepsilon(z,t)\Phi_\varepsilon(y,t)dzdy = \Lambda_3(\zeta)\Vert D^2\zeta\Vert_{\infty}^2d\cdot \varepsilon t
\end{align*}
thanks to Jensen's inequality. Thus
\begin{align*}
    A = A_1+A_2 
    \leq 
        \Lambda_H \Vert D^2\zeta\Vert_\infty (3d)\cdot \varepsilon t
        +
        \Lambda_3(\zeta)\Vert D^2\zeta\Vert_\infty^2 d\cdot\varepsilon t. 
 \end{align*}

\medskip 
\noindent 
{\it Step 2. Estimate for $B$.} By Taylor expanding $H(\cdot,p)$ at $x$, we obtain
\begin{align*}
    H(x-y, p_x) - H(x,p_x) &= D_xH(x,p_x)\cdot(-y) \\
    &+ \int_0^1 \big( y\cdot D^2_{xx}H(x-sy,p_x)\cdot y\big) (1-s)\;ds .
\end{align*}
From \ref{itm:H2}, $\int_0^1 \left( y\cdot D^2_{xx}H(x-sy,p_x)\cdot y\right) (1-s)\;ds \leq \frac{\Lambda_H}{2}|y|^2$, 
and by the symmetry of the heat kernel, 
$\int_{\R^d} D_xH(x,p_x)\cdot(-y) \Phi_\varepsilon(y,t)dy = 0$. 
Therefore
\begin{align}\label{eq:B}
    & B 
    \leq \frac{\Lambda_H}{2} \int_{\R^d} |y|^2 \Phi_\varepsilon(y,t)dy = 
    \Lambda_H d\cdot  \varepsilon t. 
\end{align}
Combining the estimates for $A$ and $B$, we obtain \eqref{eq:CommutatorUpper}.
\end{proof}

\else \fi


\subsection{Estimates for the viscous Hamilton--Jacobi equation}
Throughout this section, we assume \ref{itm:H1}--\ref{itm:H2} and \(w_0 \in C^2(\R^d)\).
Let \(w:\R^d\times [0,h]\to \R\), defined by \(w(x,t)=\mathcal{S}^{\mathrm{HJ}_\delta}_t w_0(x)\), be the solution of 
\begin{align}\label{eq:HJBdelta}
    \begin{cases}
        \begin{aligned}
            &\partial_tw  + H(x,Dw) - \delta \Delta w = 0 &&\text{in}\;\R^d\times (0,h), \\
            &w(\cdot,0) = w_0 &&\text{on}\;\R^d. 
        \end{aligned}
    \end{cases}
\end{align}
By the existence results in \cite{amann_existence_1978}, \cite[Appendix]{lions_generalized_1982}, and \cite[Sec.~1.7]{tran_hamilton-jacobi_2021}, there exists a unique solution $w$ to \eqref{eq:HJBdelta} that is sufficiently smooth, and $Dw$ and $|D^2 w|$ are bounded (possibly depending on $\delta$). To treat the unbounded domain, we apply the maximum principle to $\phi(x,t) - \kappa (1+|x|^2)^{1/2}$,
where $\phi$ denotes the quantity to be estimated. The maximum is attained, and we then let $\kappa \to 0$
(see \cite[Remark~1.27]{tran_hamilton-jacobi_2021}). 

In what follows, we establish estimates for \(Dw\), \(w_t\), and upper and lower bounds for \(D^2w\), assuming for simplicity that the maxima and minima are attained.
Given \(w_0 \in C^2(\R^d)\), we define
\begin{align}
    K_0(w_0) &= \big( e^{3\Lambda_H T}\Vert Dw_0\Vert_{L^\infty(\R^d)}^2 + \tfrac{1}{3}\left(e^{3\Lambda_H T}-1\right) \big)^{1/2} \label{eq:K0} \\ 
    K_1(w_0) &= \max \left\{ |H(x,\xi)|: |\xi|\leq K_0(w_0)\right\}.  \label{eq:K1}\\
    K_2(w_0) &= \max\{\mathrm{SC}(w_0), C_H\} \label{eq:K2}\\
    K_3(w_0) &= \inf \lbrace \delta \Delta w_0(x) - H(x,Dw_0(x)) : x\in \R^d \rbrace \label{eq:K3} \\
    K_4(w_0) &= K_4(H, T\varepsilon, \mathrm{Lip}(w_0)) \label{eq:K4}. 
\end{align}
Here $K_4(w_0)$ is the constant in \eqref{eq:uniformLower} of Proposition~\ref{prop:BoundForDelta}, depending only on $\mathrm{Lip}(w_0)$.
Proofs of Lemmas~\ref{lem:bernstein-Dzeta}--\ref{lem:bernstein-D2u-lower} are provided in the Supplemental Material.

\begin{lem}[Gradient Estimate]\label{lem:bernstein-Dzeta} 
Assume \ref{itm:H1}--\ref{itm:H2} and \(w_0 \in C^2(\R^d)\). The solution of \eqref{eq:HJBdelta} satisfies, for \(t\in[0,h]\), that 
\begin{align} \label{eq:Bound-DuNew}
        \Vert Dw(\cdot,t)\Vert_{L^\infty(\R^d)}^2
            \leq e^{3\Lambda_H t} \Vert Dw_0\Vert^2_{L^\infty(\R^d)}
        + 
            \frac{1}{3}\left(e^{3\Lambda_H t} - 1\right)
            \leq K_0(w_0).
    \end{align}
\end{lem}

\ifprintanswers 


\begin{proof}[Proof of Lemma \ref{lem:bernstein-Dzeta}] Let $\varphi(x,t) = \frac{1}{2}|Dw(x,t)|^2$ for $(x,t)\in \R^d\times [0,h]$. 
    Differentiating the equation with respect to $x_i$, then multiply both sides by $w_{x_i}$, and sum over $i=1,2,\ldots, d$ we obtain
    $$
        \varphi_t + D_xH(x,Dw)\cdot Dw + D_pH(x,Dw)\cdot D\varphi - \delta \Delta \varphi + \delta |D^2w|^2 = 0. 
    $$
By \ref{itm:H2} we have $|D_xH(x,Df)\cdot Dw|\leq \frac{1}{2}\lambda_H + 3\Lambda_H \varphi(x,t)$. 
We deduce that 
\begin{align*}
    \varphi_t -3\Lambda_H \varphi + D_pH(x,Dw)\cdot D\varphi - \delta \Delta \varphi \leq \frac{\Lambda_H}{2}. 
\end{align*}
Let $\psi(x,t) = e^{-3\Lambda_H t}\left(\varphi(x,t) - \frac{1}{6}\left(e^{3\Lambda_H t} - 1\right) \right)$ for $(x,t) \in \R^d\times [0,h]$. 
Then
\begin{equation*}
    \psi_t + D_pH(x,Dw)\cdot D\psi - \delta \Delta \psi \leq 0 \qquad\text{in}\;(0,h). 
\end{equation*}
By the maximum principle, $\sup_{x\in \R^d}\psi(x,t) \leq \sup_{x\in \R^d}\psi(x,0)$. 
This implies that
\begin{align*}
    e^{-3\Lambda_H t}\left(\frac{1}{2}|Dw(x,t)|^2 - \frac{e^{3\Lambda_H t} - 1}{6} \right)  \leq \frac{1}{2}|Dw_0(x)|^2.
\end{align*}
Hence \eqref{eq:Bound-DuNew} follows. 
\end{proof}

\else \fi

\begin{lem}[Time-derivative Estimate]
\label{lem:bernstein-Dtimezeta} 
Assume \ref{itm:H1}--\ref{itm:H2} and \(w_0 \in C^2(\R^d)\).
The solution of \eqref{eq:HJBdelta} satisfies, for \(t\in[0,h]\), 
    \begin{align}\label{eq:Bound-utNew} 
        w_t(x,t) \geq \inf_{x\in \R^d} \big\lbrace \delta\Delta w_0(x) -  H(x,Dw_0(x)) \big\rbrace. 
    \end{align}
\end{lem}

\ifprintanswers

\begin{proof}[Proof of Lemma \ref{lem:bernstein-Dtimezeta}]
The proof follows from the comparison and maximum principles; see \cite[Theorem 1.25]{tran_hamilton-jacobi_2021}, using that \(w_0\in C^2(\R^d)\).
\end{proof}

\else \fi


\begin{lem}[Semiconcave Estimate]\label{lem:bernstein-D2u-upper}
    Assume \ref{itm:H1}--\ref{itm:H2} and \(w_0 \in C^2(\R^d)\).
The solution of \eqref{eq:HJBdelta} satisfies, for \(t\in[0,h]\), 
\begin{align}\label{eq:BoundUpper-D2u}
    \mathrm{SC}(w(\cdot,t)) \leq  
    \max \left\lbrace \mathrm{SC}(w_0), C_H\right\rbrace  \mathbb{I}_d. 
\end{align}
\end{lem}

\ifprintanswers


\begin{proof}[Proof of Lemma \ref{lem:bernstein-D2u-upper}]
    Let $\xi\in \R^d$ with $|\xi|=1$. For $(x,t)\in \R^d\times (0,h)$, we write $w_{\xi}= Dw\cdot \xi$ and $w_{\xi\xi} = \langle D^2w\,\xi, \xi\rangle $. 
The semiconcavity of $w_0$ implies that $w_{\xi\xi}(x,0)\leq \mathrm{SC}(w_0)$. 
Differentiating the equation twice with respect to $\xi$, we obtain 
\begin{align*}
    & \partial_ tw_{\xi\xi} - \delta \Delta w_{\xi\xi} 
    + 
    \left(
        D_{\xi\xi} H(x,Dw) + D_{p\xi}H(x,Dw)\cdot Dw_{\xi}
    \right)\\
    &\;\; + 
    \left(
        D_pH(x,Dw)\cdot Dw_{\xi\xi}
        + 
        D_{p\xi}H(x,Dw)\cdot Dw_{\xi}
        + 
        Dw_\xi \cdot D_{pp}H(x,Dw) \cdot Dw_\xi 
    \right)
    = 0.
\end{align*}
By \ref{itm:H2} we have
\begin{align*}
    & \theta |Dw_{\xi}|^2
    \leq
    Dw_{x_i}\cdot D^2_{pp}H(x, Dw)\cdot Dw_{x_i}  \\
    & |D^2_{p\xi}H(x, Dw)\cdot Dw_{\xi}| \leq 
    \frac{\theta}{4}|Dw_{\xi}|^2 
    + 
    \frac{1}{\theta}|D^2_{p\xi}H(x,Dw)|^2 \leq \frac{\theta}{4}|Dw_{\xi}|^2 + \frac{\Lambda_H}{\theta}\\
    & |D^2_{\xi\xi} H(x,Dw)| \leq \Lambda_H. 
\end{align*}
Therefore, letting \(\phi(x,t)=w_{\xi\xi}(x,t)\), we have
\begin{align*}
    \left(
        \phi_t + D_pH(x,Dw)\cdot D\phi - \delta \Delta \phi
    \right) 
    + \theta|Dw_\xi|^2 \leq \frac{\theta}{2}|Dw_\xi|^2 + 
        \Lambda_H\left(1+\frac{2}{\theta}\right).
\end{align*}
Recalling that $w_\xi = Dw\cdot \xi$, we have
\begin{equation*}
    |Dw_\xi|^2 = |D^2w_ \xi|^2 = (D^2w  \xi)^T  (D^2w  \xi) = \xi^T \cdot (D^2w_\xi)^2\cdot  \xi \geq (\xi^T \cdot D^2w_\xi\cdot \xi)^2 = \phi^2. 
\end{equation*}
Therefore
$$
    \left(
        \phi_t + D_pH(x,Dw)\cdot D\phi - \delta \Delta \phi 
    \right) 
    + \frac{\theta}{2} \phi^2 \leq \Lambda_H\left(1+\frac{2}{\theta}\right)
$$
for $(x,t)\in \R^d\times (0,h)$. 
By the maximum principle, we conclude that
\begin{align*}
    \sup_{x\in \R^d} \phi(x,t) \leq \max \left\lbrace 
    \sup_{x\in \R^d} \phi(x,0), C_H  \right\rbrace 
    = \max \{\mathrm{SC}(w_0), C_H\} , 
\end{align*}
and thus the conclusion follows. 
\end{proof}

\else \fi

\begin{lem}[Hessian Estimate]\label{lem:bernstein-D2u-lower}
    Assume \ref{itm:H1}--\ref{itm:H2} and \(w_0 \in C^2(\R^d)\).
The solution of \eqref{eq:HJBdelta} satisfies, for \(t\in[0,h]\),
\begin{align}\label{eq:BoundLower-D2u}
    D^2w(\cdot,t) \succeq \frac{1}{\delta}\left[ K_3(w_0) - K_1(w_0) \right] - (d-1) K_2(w_0),
\end{align}
where the constants \(K_0,K_1,K_2,K_3\) are defined in \eqref{eq:K0}--\eqref{eq:K3}.
\end{lem}

\ifprintanswers

\begin{proof}[Proof of Lemma \ref{lem:bernstein-D2u-lower}] 
For all $t\in (0,h]$, from Lemma \ref{lem:bernstein-Dtimezeta} we have
\begin{align*}
    \delta \Delta w(x,t) - H\big(x,Dw(x,t)\big) 
    \geq  \inf_{x\in \R^d} \Big( \delta \Delta w(x,0) - H(x,Dw(x,0)) \Big). 
\end{align*}
Fix \((x,t)\in \mathbb{R}^d \times (0,h)\). Since \(D^2 w(x,t)\) is symmetric, its eigenvalues are real and can be ordered as \(\lambda_1(x,t)\leq \cdots \leq \lambda_d(x,t)\).
Moreover, \(\lambda_i(x,t)\leq \mathrm{SC}(w(\cdot,t))\) for all \(i=1,\ldots,d\).
From \eqref{eq:BoundUpper-D2u} we have
\begin{equation*}
    \Delta w(x,t) \leq \lambda_i(x,t) +(d-1)\max\{\mathrm{SC}(w_0), C_H\}
\end{equation*}
for all $i=1,\ldots, d$. 
As a consequence, for $i=1,\ldots, d$ then
\begin{align*}
    \lambda_i(x,t) 
    \geq 
        \Delta w - (d-1) \max\{\mathrm{SC}(w_0), C_H\}
    \geq 
        \frac{K_3(w_0) - K_1(w_0)}{\delta}- (d-1) K_2(w_0). 
\end{align*}
The conclusion follows.
\end{proof}

\else \fi

Next, we present a proof for the fact that, throughout the modified splitting scheme, at each step, the updated initial data satisfies the same first- and second-order derivative estimates as in Lemmas \ref{lem:bernstein-Dzeta}--\ref{lem:bernstein-D2u-lower}.

\begin{prop} \label{prop:BoundForDelta} Assume \ref{itm:H1}--\ref{itm:H2}, \(w_0\in C^2(\R^d)\), and for \(i=1,\ldots,n\) define 
$v_\delta(x,t_i) = \mathcal{S}^{\mathrm{H}}_{h} \circ \mathcal{S}^{\mathrm{HJ}_\delta}_{h} w(\cdot,t_{i-1})(x) = \mathcal{S}^{\mathrm{H}}_h \zeta^{(i)}_\delta(\cdot, h)(x)$ 
with $v_\delta(x,t_0) = w_0(x)$,
where
\begin{align*}
\begin{cases}
    \begin{aligned}
        \partial_t\zeta^{(i)}_\delta + H\big(x,D\zeta^{(i)}_\delta\big) &= \delta \Delta \zeta^{(i)}_\delta &&\text{in}\;\R^d\times (0,h), \\
        \zeta^{(i)}_\delta(x,0) &= v_\delta(x,t_{i-1}) &&\text{on}\;\R^d. 
    \end{aligned}
\end{cases}
\end{align*}
Recall \(K_0,K_1,K_2,K_3\) from \eqref{eq:K0}--\eqref{eq:K3}. For \(i=1,\ldots,n\), we have
\begin{align}
    & \Vert Dv_\delta(\cdot ,t_i)\Vert_{L^\infty(\R^d)}
        \leq 
        \Vert D\zeta^{(i)}_\delta(\cdot,h)\Vert_{L^\infty(\R^d)}
        \leq K_0(w_0) \label{eq:grad-bound-alli} \\[1mm]
    & \mathrm{SC}(v_\delta(\cdot, t_{i})) \leq 
         \mathrm{SC}(\zeta^{(i)}_\delta(\cdot, h)) 
         \leq 
         \max\left\lbrace \mathrm{SC}(w_0), C_H\right\rbrace = K_2(w_0) \label{eq:uniformUpper} \\
    & D^2 v_\delta(\cdot,t_i) \succeq D^2 \zeta^{(i)}_\delta(\cdot,h)  
    \succeq 
   - d\Vert D^2w_0\Vert_{L^\infty(\R^d)} - \frac{K_4(w_0)}{\delta} - (d-1) K_2(w_0),
    \label{eq:uniformLower}
\end{align}
where $K_4(w_0)$ depends only on $H, T\varepsilon, \mathrm{Lip}(w_0)$. 
More importantly, the estimates \eqref{eq:grad-bound-alli} and \eqref{eq:uniformUpper} remain valid in the limit case $\delta = 0$.
\end{prop}

\ifprintanswers

\begin{proof}[Proof of Proposition \ref{prop:BoundForDelta}] 
To obtain the gradient estimate \eqref{eq:grad-bound-alli}, we utilize \eqref{eq:Bound-DuNew} 
and induction to obtain 
\begin{align*}
     & \Vert Dv_\delta(\cdot,t_i)\Vert_{L^\infty(\R^d)}^2 
     \leq 
        \Vert D\zeta^{(i)}_\delta(\cdot, h) \Vert_{L^\infty(\R^d)}^2  \\
     &\qquad 
        \leq e^{3\Lambda_H \cdot (ih)} 
        \Vert Dv_\delta(\cdot, t_{0})\Vert_{L^\infty(\R^d)}^2 + \tfrac{1}{3}\big(e^{3\Lambda_H h} - 1\big) \left(e^{3\Lambda_H\cdot(i-1)h} + \ldots + 1\right) \\
    &\qquad \leq 
        e^{3\Lambda_H \cdot (ih)} 
        \Vert Dw_0\Vert_{L^\infty(\R^d)}^2 + \tfrac{1}{3}\left(e^{3\Lambda_H\cdot(ih)} - 1\right) \leq K_0(w_0)^2
\end{align*}
for $i=1,\ldots, n$. 
Since $nh= T$, we obtain the conclusion \eqref{eq:grad-bound-alli}. \medskip 

The concavity estimate \eqref{eq:uniformUpper} can be obtained similarly. Since $\mathcal{S}^{\mathrm{H}}_h$ is monotone, we have $\mathrm{SC}(v_\delta(\cdot, t_i)) \leq \mathrm{SC}(\zeta^{(i)}_\delta(\cdot, h))$. 
By \eqref{eq:BoundUpper-D2u} we have 
$$
    \mathrm{SC}(\zeta^{(i)}_\delta(\cdot, h)) \preceq 
    \max\{ \mathrm{SC}(\zeta^{(i)}_\delta(\cdot,0)), C_H\} = 
    \max\left\lbrace \mathrm{SC}(v_\delta(\cdot, t_{i-1})), C_H\right\rbrace. 
$$
By induction, we obtain the conclusion \eqref{eq:uniformUpper}. 
\medskip 

    Deriving \eqref{eq:uniformLower} is more involved, requiring the commutator estimate \eqref{eq:CommutatorLower} from Proposition \ref{prop:CommutatorEstimates}-(i). By \eqref{eq:Bound-utNew} of Lemma \ref{lem:bernstein-Dtimezeta}, for $t\in (0,h]$ we have
    \begin{align}\label{eq:estimate-zetat}
        \partial_t \zeta^{(i)}_\delta(x,t) \geq \inf_{x\in \R^d} 
        \{
        \delta \Delta \zeta^{(i)}_\delta(x,0)  -H(x,D\zeta^{(i)}_\delta(x,0)) \}. 
    \end{align}
To estimate \(v_\delta(\cdot,t_i)\), by \eqref{eq:changeSD} we write \(\delta \Delta v_\delta(x,t_i)-H(x,Dv_\delta(x,t_i)) = I_1+I_2\), where 
\begin{align*}
    I_1 &=\delta \Delta \big(\mathcal{S}^{\mathrm{H}}_h\zeta^{(i)}_\delta(x,h)\big) - \mathcal{S}^{\mathrm{H}}_h H(x, D\zeta^{(i)}_\delta(x,h)), \\
    I_2 &= \mathcal{S}^{\mathrm{H}}_h H(x, D\zeta^{(i)}_\delta(x,h)) 
    - H\big(x, \mathcal{S}^{\mathrm{H}}_h D\zeta^{(i)}_\delta(x,h)\big). 
\end{align*}
By \eqref{eq:grad-bound-alli}, we have $ |H\big(x,D\zeta^{(i)}(x-y,h)\big)|  \leq K_1(w_0)$. 
Using \eqref{eq:estimate-zetat} we have 
\begin{align*}
    I_1 
    &= \int_{\R^d} \Big( 
    \delta \Delta \zeta^{(i)}_\delta(x-y,h) - H(x-y, D\zeta^{(i)}_\delta(x-y, h)) \Big)\Phi_\varepsilon(y,t) \;dy \\ 
    &= \int_{\R^d} \partial_t \zeta_\delta^{(i)}(x-y,h)\Phi_\varepsilon(y,h)\;dy 
        \geq 
        \inf_{x\in \R^d} \{  
        \delta \Delta \zeta^{(i)}_\delta(x,0)  -H(x,D\zeta^{(i)}_\delta(x,0))\}\\
    &\qquad\qquad\qquad\qquad\qquad\qquad\qquad\;  = 
    \inf_{x\in \R^d} \left\lbrace  
    \delta \Delta v_\delta(x,t_{i-1})  -H(x,Dv_\delta(x,t_{i-1}))\right\rbrace. 
\end{align*}
For $I_2$, using the lower bound in the commutator estimate, namely \eqref{eq:CommutatorLower} in Proposition \ref{prop:CommutatorEstimates}-(i) we obtain $I_2 \geq -\Lambda_1(K_0(w_0))\cdot  \varepsilon h$
thanks to \eqref{eq:grad-bound-alli}.
Combining \(I_1\), \(I_2\), and using induction, we obtain for $i=1,\ldots, n$, 
\begin{align*}
    & \delta \Delta v_\delta(x, t_i) - H(x,Dv_\delta(x,t_i))   \\
    & \geq 
    \inf_{x\in \R^d} \left\lbrace  
    \delta \Delta v_\delta(x,t_{i-1})  -H(x,Dv_\delta(x,t_{i-1}))\right\rbrace
    - 
    \Lambda_1(K_0(w_0))\cdot \varepsilon h \\
    & \geq \inf_{x\in \R^d} \left\lbrace  
    \delta \Delta w_0(x)  -H(x,Dw_0(x))\right\rbrace - \Lambda_1(K_0(w_0)) \cdot T\varepsilon.
\end{align*}
In other words, we have for all $i=1,\ldots, n$ that
\begin{align*}
    K_3 (v_\delta(\cdot, t_i)) = K_3 (\zeta_\delta^{(i)}(\cdot, 0)) \geq  K_3(w_0) - \Lambda_1(K_0(w_0)) \cdot T\varepsilon. 
\end{align*}
By \eqref{eq:BoundLower-D2u} of Lemma \ref{lem:bernstein-D2u-lower} we obtain for all $i=1,2,\ldots n$ that
\begin{align*}
     D^2 v_\delta(x,t_i) 
        & \succeq D^2 \zeta^{(i)}_\delta(x,h) 
        \succeq \frac{1}{\delta}
            \left[ 
            K_3(\zeta^{(i)}_\delta(\cdot,0)) 
            -
            K_1(\zeta^{(i)}_\delta(\cdot, 0)) 
            \right] 
            - 
            (d-1)K_2(\zeta^{(i)}_\delta(\cdot,0)) \\ 
        & \succeq  \frac{1}{\delta} \left[ K_3(w_0) - \Lambda_1(K_0(w_0))\cdot T\varepsilon - K_1(w_0) \right] - (d-1)K_2(w_0).
\end{align*}
Let $K_4(w_0) = \max \{\Lambda_1(K_0(w_0)) T\varepsilon, K_1(w_0)\}$ depends only on $H,T\varepsilon$ and $\mathrm{Lip}(w_0)$. 
We deduce the conclusion.
\end{proof}

\else \fi

We record the optimal upper 
rate for vanishing viscosity; the proof is given in the Supplement Materials.

\begin{cor} \label{cor:comparingUpper}
Assume \ref{itm:H1}--\ref{itm:H2}.  
Let \(w\) and \(u\) denote the solutions to \eqref{eq:HJBdelta} with initial data 
\(w_0 \in C^2(\R^d)\) and \(u_0 \in W^{1,\infty}(\R^d)\), corresponding to 
\(\delta>0\) (the viscous Hamilton--Jacobi equation) and \(\delta=0\) 
(the first-order Hamilton--Jacobi equation), respectively.
For $t\in [0,h]$ we have
\begin{align*}
    \sup_{x\in \R^d}\big( w (x,t) - u(x,t)\big) \leq \Vert (w_0-u_0)^+\Vert_{\infty} +  \max\{\mathrm{SC}(w_0), C_H\}\cdot dh.
\end{align*}
\end{cor}

\ifprintanswers

\begin{proof}[Proof of Corollary \ref{cor:comparingUpper}] By Lemma~\ref{lem:bernstein-D2u-upper}, we have the Hessian (semiconcave) estimate $\Delta  w(x,t) \leq d\max\{\mathrm{SC}(w_0),C_H\}$. For $(x,t) \in \R^d\times [0,h]$ we define
$$
    \varphi(x,t) = w(x,t)- \max\{\mathrm{SC}(w_0),C_H\}\cdot dh - \Vert (u_0-v_0)^+\Vert_{L^\infty(\R^d)}. 
$$
Then $\varphi_t + H(x,D\varphi) \leq 0$ in $\R^d\times (0,h)$, 
and $\varphi(x,0) 
\leq u_0(x)$. By the comparison principle for the first-order equation \eqref{eq:HJBdelta} with $\delta=0$ we obtain $\varphi(x,t)\leq u(x,t)$ for all $(x,t)\in \R^d\times [0,h]$. The conclusion thus follows. 
\end{proof}

\else \fi

\subsection{Splitting Error in $L^\infty(\R^d)$} 
We present the error analysis for the splitting scheme \eqref{eq:split-scheme}. 
Let $T = nh$ and $t_i = ih$ for $i = 0,1,\ldots,n$.
For convenience, we recall the definitions of the true solution $u$, the splitting approximation $v$ introduced in the introduction, and the new object, the regularized splitting function $v_\delta$, as follows. \medskip


\noindent 
{\it 1. The true solution.} Let $u_0 \in W^{1,\infty}(\R^d)$. Let $u$ be the exact solution to \eqref{eq:forwardtemp}. 


\noindent 
{\it 2. The splitting function.} Let $v_0\in W^{1,\infty}(\R^d)$ and for $i=1,2,\ldots, n$, we define 
\begin{align}\label{eq:onevRd}
\begin{cases}
\begin{aligned}
    v(x,t_i) &= \mathcal{S}^{\mathrm{H}}_{h} \circ \mathcal{S}^{\mathrm{HJ}}_{h} v(\cdot, t_{i-1})(x) = \mathcal{S}^{\mathrm{H}}_{h} \zeta^{(i)}(\cdot, h)(x), \qquad x\in \R^d, \\
    v(x,t_0) &= v_0(x).
\end{aligned}
\end{cases}\\
\text{where}\qquad 
\begin{cases} 
    \begin{aligned}
        \partial_t\zeta^{(i)} + H\big(x,D\zeta^{(i)}\big) &= 0 &&\quad \text{in}\;\R^d\times (0,h), \\
        \zeta^{(i)}(x,0) &= v(x,t_{i-1}) &&\quad \text{on}\;\R^d. 
    \end{aligned}
\end{cases} \nonumber 
\end{align}


\noindent 
{\it 3. The regularized splitting function.} Let $w_0\in C^2(\R^d)$, and for $i=1,\ldots, n$, we define 
\begin{align}\label{eq:vdelta}
\begin{cases}
\begin{aligned}
    v_\delta(x,t_i) &= 
      \mathcal{S}^{\mathrm{H}}_{h} \circ \mathcal{S}^{\mathrm{HJ}_\delta}_{h} v_\delta(\cdot,t_{i-1})(x) = \mathcal{S}^{\mathrm{H}}_h \zeta^{(i)}_\delta(\cdot, h)(x), \qquad x\in \R^d, \\
    v_\delta(x,t_0) &= w_0(x)
\end{aligned}
\end{cases}  \\ 
\text{where}\qquad 
\begin{cases}
    \begin{aligned}
        \partial_t\zeta^{(i)}_\delta + H\big(x,D\zeta^{(i)}_\delta\big) &= \delta \Delta \zeta_\delta^{(i)} &&\quad \text{in}\;\R^d\times (0,h), \\
        \zeta_\delta^{(i)}(x,0) &= v_\delta(x,t_{i-1}) &&\quad \text{on}\;\R^d. 
    \end{aligned}
\end{cases} \label{eq:KeyZetaDelta}
\end{align}
Recall \(K_0,K_1,K_2,K_3\) from \eqref{eq:K0}--\eqref{eq:K3}. 
The proofs of the following Propositions \ref{prop:key-result-lower}--\ref{prop:key-result} are provided in Supplemental Materials. 

\begin{prop}[Lower bound of $u-v$]\label{prop:key-result-lower} Assume \ref{itm:H1}--\ref{itm:H2}. The true solution \(u\) of \eqref{eq:forwardtemp} and the splitting solution \(v\) of \eqref{eq:onevRd} satisfy 
    \begin{align}\label{eq:LowerAll}
        \Vert (v-u)^+(\cdot, t_i) \Vert_{\infty} 
        \leq 
        \Vert (v_0-u_0)^+ \Vert_{\infty}  
        + 
        \Lambda_1(v_0) \cdot T\varepsilon h,
    \end{align}
    for all \(i=1,\ldots,n\), 
    where $\Lambda_1(v_0) = \Lambda_1(H, \mathrm{Lip}(v_0))$. 
\end{prop}

\ifprintanswers


\begin{proof}[Proof of Proposition \ref{prop:key-result-lower}]
By \eqref{eq:grad-bound-alli} in Proposition \ref{prop:BoundForDelta} with \(\delta=0\), we have 
$$
    \Vert Dv(\cdot, t_i)\Vert_{L^\infty(\R^d)} \leq  \Vert D\zeta^{(i)}(\cdot, h)\Vert_{L^\infty(\R^d)} \leq K_0(v_0), 
$$
for \(i=1,\ldots,n\), where $K_0(w_0)$ is given by \eqref{eq:K0}. 
By the commutator estimate \eqref{eq:CommutatorLower} in Proposition \ref{prop:CommutatorEstimates}, there exists a constant $\Lambda_1(v_0) = \Lambda_1(H, \mathrm{Lip}(v_0))$
such that 
\begin{align}\label{eq:Lambda1est}
    \mathcal{S}_t^{\mathrm{H}} H(x,D\zeta^{(i)}(x,t)) 
    - H\big(x,D\mathcal{S}^{\mathrm{H}}_t \zeta^{(i)}(x,t)\big) 
    \geq  - \Lambda_1(v_0)\cdot \varepsilon t.
\end{align}
We introduce the comparison function (see, e.g., \cite{jakobsen_convergence_2001})
    \begin{equation*}
        q^{(i)}(x,t) = \mathcal{S}_t^{\mathrm{H}}\circ \mathcal{S}^{\mathrm{HJ}}_t v(\cdot, t_{i-1})(x) =\mathcal{S}_t^{\mathrm{H}} \zeta^{(i)}(x,t), \qquad (x,t) \in \R^d\times [0,h]. 
    \end{equation*}
    We have $q^{(i)}(\cdot,0) = v(\cdot, t_{i-1})$ and  $q^{(i)}(\cdot, h) = v(x,t_i)$. 
    By Lemma \ref{lem:diffS}, we have
    \begin{align*}
        & \partial_t q^{(i)}(x,t) = \partial_t \big(\mathcal{S}^{\mathrm{H}}_t \zeta^{(i)}(x,t)\big) \\ 
        &\qquad = \mathcal{S}^{\mathrm{H}}_t \partial_t \zeta^{(i)}(x,t) + \varepsilon \Delta q^{(i)}(x,t) 
        = - \mathcal{S}^{\mathrm{H}}_tH(x,D\zeta^{(i)}(x,t)) + \varepsilon \Delta q^{(i)}(x,t)
    \end{align*}
    thanks to $\partial_t \zeta^{(i)}(x,t) = -H(x,D\zeta^{(i)}(x,t))$ for a.e. $(x,t)\in \R^d\times [0,h]$. Therefore
    \begin{align*} 
        & \partial_t q^{(i)}(x,t) + H(x,Dq^{(i)}(x,t)) - \varepsilon\Delta q^{(i)}(x,t) \\
        & \qquad \qquad \;\, 
        = 
        H(x,D\mathcal{S}^{\mathrm{H}}_t\zeta^{(i)}(x,t)) - \mathcal{S}_t^{\mathrm{H}} H(x,D\zeta^{(i)}(x,t)) \leq \Lambda_1(v_0)\cdot \varepsilon t. 
    \end{align*}
thanks to \eqref{eq:Lambda1est}. 
By the comparison principle, we deduce that
\begin{align*}
    q^{(i)}(x,h) - u(x,t_i) 
    \leq 
    \Vert (q^{(i)}(x,0) - u(x,t_{i-1}) )^+\Vert_{L^\infty(\R^d)} + \frac{1}{2}\Lambda_1(v_0)\cdot \varepsilon h^2 , \qquad x \in \R^d. 
\end{align*}
In other words, we have
\begin{align*}
     \Vert (v-u)^+(\cdot, t_i) \Vert_{L^\infty(\R^d)} 
        \leq 
    \Vert (v-u)^+(\cdot, t_{i-1}) \Vert_{L^\infty(\R^d)} 
        + \Lambda_1(v_0)\cdot \varepsilon h^2
\end{align*}
and thus the conclusion \eqref{eq:LowerAll} follows from induction. 
\end{proof}

\else \fi

\begin{prop}[Upper bound of $u-v$]\label{prop:key-result}
Assume \ref{itm:H1}--\ref{itm:H2}. For any \(\delta>0\) and \(w_0\in C^2(\R^d)\), the solution \(u\) of \eqref{eq:forwardtemp} and the splitting solution \(v\) of \eqref{eq:onevRd} satisfy
\begin{align}\label{eq:UpperAll}
        & 
        \Vert (u(\cdot, t_i) - v(\cdot, t_i))^+ \Vert_{\infty}
        \leq 
        \Vert (u_0-w_0)^+\Vert_{\infty} 
         + 
        \Vert(w_0-v_0)^+\Vert_{\infty} \nonumber \\
        &
        \quad  + CT\left(1+ \mathrm{SC}(w_0)^2 + \Vert D^2w_0\Vert_\infty^2 + \frac{1}{\delta^2} \right)\cdot \varepsilon h
        + 
        \max\{\mathrm{SC}(w_0), C_H\}dT\cdot \delta, 
\end{align}
where $C=C(H,d,\varepsilon,T, \mathrm{Lip}(w_0)$. 
\end{prop}

\ifprintanswers

\begin{proof}[Proof of Proposition \ref{prop:key-result}] The key idea is to use the regularized splitting function $v_\delta$ defined in \eqref{eq:vdelta}. We write $u - v = (u - v_\delta) + (v_\delta - v)$. 

\medskip 
\noindent
{\it Step 1. Estimates for $v_\delta - v$.} 
At step $i$, by Corollary \ref{cor:comparingUpper} we have
    \begin{align*}
        & v_\delta(x,t_i) - v(x,t_i) 
        = 
        \mathcal{S}^{\mathrm{H}}_h\big( \zeta^{(i)}_\delta(x,t_i) - \zeta^{(i)}(x,t_i) \big) \leq 
        \zeta^{(i)}_\delta(x,t_i) - \zeta^{(i)}(x,t_i) \nonumber  \\
        &\qquad \qquad 
        \leq 
        \Vert \big( 
            \zeta^{(i)}_\delta(\cdot ,0) - \zeta^{(i)}(\cdot,0) 
            \big) ^+ 
        \Vert_{L^\infty(\R^d)} + \max\{\mathrm{SC}(v_\delta(\cdot,t_{i-1})), C_H\}d \cdot \delta h \nonumber \\
        &\qquad \qquad = \Vert \big( 
            v_\delta(x,t_{i-1}) - v(x,t_{i-1})
            \big) ^+ 
        \Vert_{\infty} + \max\{\mathrm{SC}(w_0), C_H\}d \cdot  \delta h \nonumber \\
        &\qquad \qquad \leq \Vert (w_0-v_0)^+\Vert_{\infty} 
        + 
        \max\{\mathrm{SC}(w_0), C_H\}d T\cdot \delta.
    \end{align*}
    where we used \(\mathrm{SC}(v_\delta(\cdot,t_i)) \le \max\{\mathrm{SC}(w_0), C_H\}\) by \eqref{eq:uniformUpper} in Proposition~\ref{prop:BoundForDelta}.
    

\medskip
\noindent
{\it Step 2. Estimates for $u-v_\delta$.} Let \(\zeta_\delta^{(i)} : \R^d \times [0,h] \to \R\) solve \eqref{eq:KeyZetaDelta}. We define
    \begin{align*}
        q^{(i)}_\delta(x,t) := \mathcal{S}^{\mathrm{H}}_{t} \circ \mathcal{S}^{\mathrm{HJ}_\delta}_{t} v_\delta(\cdot,t_{i-1})(x) = \mathcal{S}^{\mathrm{H}}_{t} \zeta_\delta^{(i)}(x,t),  
    \end{align*}
    We have $q_\delta^{(i)}(\cdot,0) = v_\delta(\cdot, t_{i-1})$ and $q_\delta^{(i)}(\cdot,h) = v_\delta(\cdot, t_{i})$. 
    By Proposition \ref{prop:BoundForDelta} we have $ \Vert D\zeta_\delta^{(i)}(\cdot,t)\Vert_{L^\infty(\R^d)} \leq K_0(w_0)$ and 
    \begin{align*}
        -d\Vert D^2w_0\Vert_{\infty} - \frac{K_4(w_0)}{\delta}  -d K_2(w_0) \preceq D^2\zeta_\delta^{(i)}(\cdot, t)
        \preceq K_2(w_0). 
    \end{align*}
    We conclude that there exists \(C=C(H,d,T\varepsilon)\) such that
    $$
        1 + \Vert D^2\zeta^{(i)}_\delta\Vert_{\infty} \leq K_5(w_0) := C\left(1+ \mathrm{SC}(w_0) + \Vert D^2w_0\Vert_{\infty}  + \frac{\mathrm{Lip}(w_0)}{\delta}\right). 
    $$
    By \eqref{eq:CommutatorUpper} of Proposition \ref{prop:CommutatorEstimates}, there exists $\Lambda_2(w_0)=\Lambda_2(H, \mathrm{Lip}(w_0))$ such that
    \begin{align*}
    & \partial_t q_\delta^{(i)}  + H(x,Dq_\delta^{(i)} - \varepsilon\Delta q_\delta^{(i)} \\
    &= H(x,D\mathcal{S}^{\mathrm{H}}_t\zeta_\delta^{(i)}(x,t)) - \mathcal{S}_t^{\mathrm{H}} H(x,D\zeta_\delta^{(i)}(x,t)) 
    \geq 
    \Lambda_2(w_0)
    \big(
        1+ \Vert D^2\zeta^{(i)}_\delta\Vert^2_\infty  
    \big) d \cdot \varepsilon t. 
\end{align*}
By the comparison principle for \eqref{eq:forwardtemp}, noting that $q_\delta^{(i)}(\cdot,0) = v_\delta(\cdot, t_{i-1})$ we obtain
\begin{align*}
    u(x,t_i) - q_\delta^{(i)}(x,h) 
    \leq 
    \Vert (u(\cdot, t_{i-1})-q_\delta^{(i)}(\cdot, 0))^+\Vert_{\infty}
    +
    \Lambda_2(w_0)
    \Big(
        1+ \Vert D^2\zeta^{(i)}_\delta\Vert^2_\infty  
    \Big)  \cdot \varepsilon h^2.
\end{align*}
By induction, we have
\begin{align*}
    \Vert  ( u(\cdot ,t_i) - v_\delta(\cdot,t_i) )^+ \Vert _{\infty}
        &\leq 
    \Vert (u(\cdot, t_{i-1})-v_\delta(\cdot, t_{i-1}))^+\Vert_{\infty}
    + 
    \Lambda_2(w_0) K_5(w_0)^2  \cdot \varepsilon h^2 \\ 
    &\leq 
    \Vert (u_0-w_0)^+\Vert_{\infty} + \Lambda_2(w_0) K_5(w_0)^2 T \cdot \varepsilon h
\end{align*}

\noindent
{\it Step 3. Combine the estimates.} From Step 1 and Step 2, we obtain \eqref{eq:UpperAll}.
\begin{align*}
    & \Vert (u(\cdot, t_i) - v(\cdot, t_i))^+ \Vert_{\infty} 
    \leq 
    \Vert (u(\cdot, t_i) - v_\delta(\cdot, t_i))^+ \Vert_{\infty}
    +
    \Vert (v_\delta(\cdot, t_i) - v(\cdot, t_i))^+ \Vert_{\infty}\\
    &\qquad\qquad \qquad \qquad 
        \leq \Vert (u_0-w_0)^+\Vert_{L^\infty(\R^d)} + \Vert(w_0-v_0)^+\Vert_{\infty} \\
    &\qquad\qquad \qquad \qquad  
    + 
        \max\{\mathrm{SC}(w_0), C_H\}dT \cdot \delta \\
    &\qquad\qquad \qquad \qquad 
    + 
        C\Lambda_2(H)\left(1+ \mathrm{SC}(w_0) + \Vert D^2w_0\Vert_{\infty}  + \frac{\mathrm{Lip}(w_0)}{\delta}\right) ^2 T\cdot \varepsilon h.
\end{align*}
The conclusion thus follows. 
\end{proof}

\else \fi

The following chain rule lemma is used in Proposition~\ref{prop:key-result}; its proof is given in the Supplemental Material.

\begin{lem}\label{lem:diffS} Let $\varphi \in \mathrm{Lip}(\mathbb{R}^d \times [0,t])$ and 
$\tilde{\varphi} := \mathcal{S}^{\mathrm{H}}_t \varphi$. Then
$
    \partial_t \tilde{\varphi} = \mathcal{S}^{\mathrm{H}}_{t} \varphi_t +  \varepsilon\Delta \tilde{\varphi} 
$
in $\R^d\times (0,h)$. 
\end{lem}

\ifprintanswers

\begin{proof}[Proof of Lemma \ref{lem:diffS}]
    For $t\in (0,h)$ and $|\delta|$ small such that $t,t+\delta\in (0,h)$, we have
\begin{align*}
\frac{\tilde{\varphi}(x,t+\delta) - \tilde{\varphi}(x,t)}{\delta}
  &= 
  \frac{\mathcal{S}^{\mathrm{H}}_{t+\delta}\varphi(x,t+\delta) - \mathcal{S}^{\mathrm{H}}_{t+\delta}\varphi(x,t)}{\delta} 
  + 
  \frac{\mathcal{S}^{\mathrm{H}}_{t+\delta}\varphi(x,t) - \mathcal{S}^{\mathrm{H}}_{t}\varphi(x,t)}{\delta}  \\
  &= \int_{\R^d} \Phi_\varepsilon(x-y,t+\delta)\left(\frac{\varphi(y,t+\delta) - \varphi(y)}{\delta}\right)dy \\
  &+ \int_{\R^d} \left(\frac{\Phi_\varepsilon(y, t+\delta) - \Phi_\varepsilon(y, t)}{\delta}\right)\varphi(x-y,t)\;dy.
\end{align*}
For the second term, it is clear that
\begin{align*}
    &\lim_{\delta\to 0}\int_{\R^d} \left(\frac{\Phi_\varepsilon(y, t+\delta) - \Phi_\varepsilon(y, t)}{\delta}\right)\varphi(x-y,t)\;dy = \int_{\R^d} \partial_t \Phi_\varepsilon(y,t)\varphi(x-y,t)\;dy \\
    &\qquad\qquad= \int_{\R^d} \varepsilon \Delta \Phi_\varepsilon(y,t) \varphi(x-y,t)\;dy = \varepsilon\Delta \left(\int_{\R^d} \Phi_\varepsilon(y,t)\varphi(x-y,t)\right)
    = \varepsilon \Delta \tilde{\varphi}(x,t). 
\end{align*}
For the first term, we observe that for every $t\in (0,h)$, the set 
\begin{equation*}
    E_t = \{y\in \R^d: \partial_t \varphi(y,t)\;\text{does not exists}\} 
\end{equation*}
is a null set in $\R^d$. This comes from the fact that, as $\varphi\in \mathrm{Lip}(\R^d\times [0,h])$, $\varphi$ is differentiable in space-time for a.e. $(y,t)\in \R^d\times (0,h)$ by Radamacher's Theorem. Hence
\begin{equation*}
    E = \{(y,t)\in \R^d\times (0,h): \partial_t \varphi(y,t)\;\text{does not exists}\}
\end{equation*}
is a null set in $\R^d\times [0,h]$. By Fubini-Tonelli Theorem we have
\begin{equation*}
    |E| = \int_{\R^d\times [0,h]} \mathbf{1}_E(y,t)\;dydt = \int_0^h \left(\int_{\R^d} \mathbf{1}_E(y,t)\;dy\right)dt = \int_0^h |E_t|dt.
\end{equation*}
Hence $E_t$ is a null set in $\R^d$ for every $t$. We therefore can use the Dominated Convergence Theorem to obtain
\begin{align*}
    &\lim_{\delta \to 0} \int_{\R^d} \Phi_\varepsilon(x-y,t+\delta)\left(\frac{\varphi(y,t+\delta) - \varphi(y,t)}{\delta}\right)dy  \\
    &\qquad\qquad \qquad  = \int_{\R^d}\Phi_\varepsilon(x-y,t) \varphi_t(y,t)\;dy = \mathcal{S}^{\mathrm{H}}_t \varphi(x,t). 
\end{align*}
The conclusion follows.
\end{proof}

\else \fi 

\subsection{Proof of Theorem \ref{thm:LinftyA}}

\begin{proof}[Proof of Theorem \ref{thm:LinftyA}] 
  The proof of (i) follows directly from \eqref{eq:LowerAll} 
  in Proposition \ref{prop:key-result} with $u_0 \equiv v_0$. 
  For (ii), let $\rho \in C^\infty_c(\R^d)$ be a mollifier with $\int_{\R^d}\rho\;dx = 1, \rho \geq 0$ and $\rho_\kappa(x) = \kappa^{-d}\rho(\kappa^{-1}x)$. 
We define $w_0^\kappa = \rho_\kappa * u_0 \in C^\infty(\R^d)$ with 
    \begin{align*}
        \mathrm{Lip}(w^\kappa_0) \leq \mathrm{Lip}(u_0), 
        \quad 
        \Vert D^2 w^\kappa_0\Vert_{\infty} \leq \frac{C}{\kappa}\mathrm{Lip}(u_0) , 
        \quad \Vert w^\kappa_0 - u_0\Vert_{\infty} \leq C\mathrm{Lip}(u_0)\kappa, 
    \end{align*}
    where $C = \int_{\R^d}|x|\rho(x)\;dx + \int_{\R^d} D\rho(x)\;dx $. 
    \medskip 
    
    If $u_0 \in \mathrm{Lip}(\R^d)$, applying \eqref{eq:UpperAll} in Proposition~\ref{prop:key-result} with $u_0 \equiv v_0$ and $w_0 = w_0^\kappa$, and using $\mathrm{SC}(w_0^\kappa) \leq \Vert D^2 w_0^\kappa\Vert_\infty \leq C\kappa^{-1}\mathrm{Lip}(u_0)$ yields
    \begin{align*}
         \Vert (u(\cdot, t_i) - v(\cdot, t_i))^+  \Vert_{\infty}  
         &\leq 
         2C\kappa
         + 
         CT \left(1 + \frac{1}{\kappa^2
         } + \frac{1}{\delta^2} \right)\varepsilon h + CT\left(1+ \frac{1}{\kappa}\right) \delta \\
         &\leq 
         2C\kappa
         + 
         CT \left(\varepsilon h + \frac{\varepsilon h}{\kappa^2
         } + \frac{\varepsilon h}{\delta^2} + \delta + \frac{\delta}{\kappa} \right) \leq C(T\varepsilon h)^{1/5}
    \end{align*}
    by choosing $\kappa = (T\varepsilon h)^{1/5}$ and $\delta = \kappa^2$, 
    where $C$ depends on $H,d,T\varepsilon$ and $\mathrm{Lip}(u_0)$. 
    \medskip
    
    If $u_0$ is Lipschitz and semiconcave, then $\mathrm{SC}(w_0^\kappa)\leq \mathrm{SC}(u_0)$, thus \eqref{eq:UpperAll} in Proposition~\ref{prop:key-result} yields
    \begin{align*}
         \Vert (u(\cdot, t_i) - v(\cdot, t_i))^+  \Vert_{\infty}  
         &\leq 
         2C\kappa
         + 
         CT \left(1 + \frac{1}{\kappa^2
         } + \frac{1}{\delta^2} \right)\varepsilon h + CT \delta \\
         &\leq 
         2C\kappa
         + 
         CT \left(\varepsilon h + \frac{\varepsilon h}{\kappa^2
         } + \frac{\varepsilon h}{\delta^2} + \delta\right) \leq C(T\varepsilon h)^{1/3}
    \end{align*}
    by choosing $\kappa = \delta = (T\varepsilon h)^{1/3}$, 
    where $C$ depends on $H$,$d$, $T\varepsilon$, $\mathrm{Lip}(u_0)$, and $\mathrm{SC}(u_0)$. 
\end{proof}

\subsection{BV Functions and Commutator Estimates in Weighted $L^1(\R^d)$} 
We refer to \cite{cannarsa_semiconcave_2004, evans_measure_2015} 
and Supplementary Material \ref{Supp:BVfunctions} for more details on BV functions.

\begin{lem}\label{lem:BV}
    Let $\omega \in C^1(\R^d)\cap L^\infty(\R^d)$ be a weight decaying at infinity such that
    \begin{align}\label{eq:weight}
        \omega\geq 0,\qquad 
        \Vert \omega\Vert_{L^1(\R^d)}=1, 
        \qquad  D\omega \in L^1(\R^d).
    \end{align} 
    If $u\in \mathrm{BV}_{\mathrm{loc}}(\R^d)\cap L^\infty_{\mathrm{loc}}(\R^d)$, then $\omega u \in \mathrm{BV}(\R^d)$ with the total variation measure
    \begin{align}\label{eq:totalVar}
        \Vert D(\omega u)\Vert (\R^d) \leq \int_{\R^d} |\omega(x)| d\Vert Du \Vert + \int_{\R^d} |u(x)||D\omega (x)|\;dx. 
    \end{align}
    If $u\in L^\infty(\R^d)\cap \mathrm{BV}_{\mathrm{loc}}(\R^d)$, then with $\Lambda_\omega(\zeta) = \|D(\omega u)\|(\R^d) + \|u\|_\infty \|D\omega\|_{L^1(\R^d)}$, 
\begin{align}\label{eq:estimate-translation-BV}
     & \int_{\R^d} |u(x+h) - u(x)|\omega(x)dx 
     \leq  \Lambda_\omega(\zeta)\cdot |h|. 
\end{align} 
\end{lem}

\ifprintanswers 

\begin{proof}[Proof of Lemma \ref{lem:BV}] We refer to \cite{evans_measure_2015} for \eqref{eq:totalVar}. For \eqref{eq:estimate-translation-BV}, we compute
\begin{align*}
     & \int_{\R^d} |u(x+h) - u(x)|\omega(x)dx \\ 
     &\quad  \leq 
     \int_{\R^d} |(u\omega)(x+h) - (u\omega)(x)|dx 
     + 
     \int_{\R^d} |u(x+h)|\cdot|\omega(x+h)-\omega(x)|dx \\
     &\quad \leq 
     |h|\cdot \Vert D(u\omega)\Vert(\R^d) 
     + 
     |h|\int_{\R^d} |u(y)|\left(\int_0^1 |D\omega(y + (t-1)h)\;dt\right) \;dy  \\
     &\quad \leq 
     |h|\cdot \Vert D(\omega u)\Vert(\R^d) +
     |h| \cdot \Vert u\Vert_{\infty} \Vert D\omega\Vert_{L^1(\R^d)} 
     \leq  \Lambda_\omega(\zeta)\cdot |h|. 
\end{align*} 
The conclusion follows.
\end{proof}

\else \fi 

\begin{cor}\label{cor:BV-Semiconcave} If $\zeta\in \mathrm{Lip}(\R^d)$ and $D^2\zeta \preceq \mathrm{SC}(\zeta)\,\mathbb{I}_d$. Let $\omega$ satisfy \eqref{eq:weight}. Then $\omega D\zeta \in \mathrm{BV}(\R^d)$ with the total variation measure
    \begin{align}\label{eq:totalVarDzeta-omega}
        \Vert D(\omega D\zeta)\Vert (\R^d) \leq 2d\cdot \mathrm{SC} + 2\mathrm{Lip}(\zeta)\cdot \Vert \omega\Vert_{L^1(\R^d)}. 
    \end{align}
\end{cor}

\ifprintanswers 

\begin{proof}[Proof of Corollary \ref{cor:BV-Semiconcave}] From \eqref{eq:totalVar} of Lemma \ref{lem:BV} we have 
\begin{align*}
        \Vert D(\omega D\zeta)\Vert (\R^d) 
        \leq 
        \int_{\R^d} |\omega(x)| d\Vert D^2\zeta \Vert + \int_{\R^d} |D\zeta(x)||D\omega (x)|\;dx. 
\end{align*}
The second term is straight forward since $ \int_{\R^d} |D\zeta(x)||D\omega (x)|\;dx \leq \mathrm{Lip}(\zeta)\Vert D\omega\Vert_{L^1(\R^d)}$. 
For the first term, we mollify $\zeta$ by $\zeta_\delta = \zeta * \eta_\delta$, where $\eta_\delta$ is a standard mollifier; semiconcavity is preserved: $\mathrm{SC}(\zeta * \eta_\delta) \preceq \mathrm{SC}(\zeta)$. 
Let $\lambda^\delta_1(x),\ldots, \lambda^\delta_d(x)$ be the eigenvalue of the Hessian $D^2\zeta(x)$, then $\lambda_i (x) \leq \mathrm{SC}(\zeta)$ for $i=1,\ldots, d$, and 
\begin{align*}
    \sum_{i=1}^d |\lambda^\delta_i(x)| = \sum_{i=1}^d 
    \Big( 
        2(\lambda^\delta_i(x))^+ - \lambda^\delta_i(x) 
    \Big) 
    \leq 
        2d\cdot \mathrm{SC}(\zeta) - \Delta \zeta_\delta(x). 
\end{align*}
    Hence, since $\omega\geq 0$ and $\int_{\R^d} |\omega(x)|dx = 1$, we have
    \begin{align*}
        & \int_{\R^d} |\omega(x)|d\Vert D^2\zeta_\delta\Vert  
        \leq 2d\cdot\mathrm{SC}(\zeta) - \int_{\R^d} \omega(x)\Delta\zeta_\delta(x)\;dx   \\
        &\qquad =2d\cdot\mathrm{SC}(\zeta) - \int_{\R^d} D\omega(x)\cdot D\zeta_\delta(x)\;dx  
        \leq 2d\cdot \mathrm{SC} (\zeta)+ \mathrm{Lip}(\zeta)\cdot \Vert \omega\Vert_{L^1(\R^d)}, 
    \end{align*}
    thanks to the fact that $\omega$ decays at infinity. Let $\delta\to 0$ we obtain the conclusion. 
\end{proof}

\else \fi 

\begin{prop}[Commutator Estimates in $L^1$] \label{prop:commutatorL1}
Assume \ref{itm:H1}--\ref{itm:H2} and $\omega$ satisfies \eqref{eq:weight}.
Let \(\zeta \in \mathrm{Lip}(\mathbb{R}^d)\) be semiconcave.
Then we have
$$
    -\Lambda_1(\zeta) \cdot \varepsilon t
    \leq
    \int_{\R^d} \left(\mathcal{S}^{\mathrm{H}}_t \left(x,D\zeta(x)\right) - H(x,\mathcal{S}^{\mathrm{H}}_t D\zeta(x)) \right) \omega(x)dx \leq \Lambda^{\mathrm{BV}}_2(\zeta, \omega)\sqrt{\varepsilon t}
$$
where $\Lambda_1(\zeta)$ is from \eqref{eq:CommutatorLower} in Proposition \ref{prop:CommutatorEstimates}, and $\Lambda^{\mathrm{BV}}_2(\zeta,\omega)=C(1+\Lambda_\omega(\zeta))$ with $C=C(d,H,\mathrm{Lip}(\zeta))$, 
and $\Lambda_\omega(\zeta)$ is from Lemma \ref{lem:BV}. 
\end{prop}

\ifprintanswers 

\begin{proof}[Proof of Proposition \ref{prop:commutatorL1}]
If $\zeta$ is Lipschitz and semiconcave then $D\zeta\in \mathrm{L}^\infty(\R^d)\cap \mathrm{BV}_{\mathrm{loc}}(\R^d)$. 
By Proposition \ref{prop:CommutatorEstimates}, the lower bound still holds with $\Lambda_1(\zeta,\omega) = -\Lambda_1(\zeta)$, as in \eqref{eq:CommutatorLower}.

For the upper bound, we follow Proposition \ref{prop:CommutatorEstimates}, decompose the commutator as $A+B$ in \eqref{eq:comm} with $A=A_1+A_2$, and estimate each term. Using \eqref{eq:A1} and arguing as in Step~1 there, we obtain
\begin{align*}
    &\int_{\R^d} |A_1(x,t)|\omega(x)dx  \\
    &\qquad\leq
    \Lambda_H 
    \iint_{\R^d\times \R^d}  |y|\left( \int_{\R^d} |D\zeta(x-y) - D\zeta(x)|\omega(x)dx \right) \Phi_\varepsilon(y,t)\Phi_\varepsilon(z,t)dydz \\ 
    &\qquad +  
    \Lambda_H 
    \iint_{\R^d\times \R^d}  |y|\left( \int_{\R^d} |D\zeta(x) - D\zeta(x-z)|\omega(x)dx \right) \Phi_\varepsilon(y,t)\Phi_\varepsilon(z,t)dydz \\ 
    &\qquad \leq  \Lambda_H \cdot 2 \Lambda_\omega(\zeta)
    \left( 
        \iint_{\R^d\times \R^d}\left(  |y|^2  + |yz|\right) \Phi_\varepsilon(z,t)\Phi_\varepsilon(y,t)dzdy 
    \right) \\
    &\qquad \leq \Lambda_H \cdot 2 \Lambda_\omega(\zeta)\cdot ( 2d + \mathbf{C}_d^2 ) \cdot \varepsilon t
\end{align*}
where $\Lambda_\omega(\zeta)$ depends on $\mathrm{Lip}(\zeta)$ and $\omega$ by Lemma \ref{lem:BV}. On the other hand, by a similar reasoning we have 
\begin{align*}
    \int_{\R^d} |D\zeta(x-y) - D\zeta(x-z)|^2 \omega(x)\;dx \leq 2\,\mathrm{Lip}(\zeta) \cdot (|y|+|z|) \cdot 2\Lambda_\omega(\zeta). 
\end{align*}
Using \eqref{eq:A2} for $A_2$, we have
\begin{align*}
    & \int_{\R^d} |A_2(x,t)|\omega_R(x)dx \\
    & \qquad 
    \leq 
      \frac{\Lambda_3(\zeta)}{2} \int_{\R^d}\int_{\R^d} \left(\int_{\R^d} |D\zeta(x-y)-D\zeta(x-z)|^2 \omega(x)dx\right) \Phi_\varepsilon(z,t)\Phi_\varepsilon(y,t)dzdy \\
    & \qquad 
    \leq 
    \frac{\Lambda_3(\zeta)}{2}\cdot 2 \,\mathrm{Lip}(\zeta)\cdot  2\Lambda_\omega(\zeta)\int_{\R^d}\int_{\R^d} (|y| + |z|)\Phi_\varepsilon(y,t)\Phi_\varepsilon(z,t)dydz \\
    &\qquad 
    \leq 
    \frac{\Lambda_3(\zeta)}{2}\cdot 2 \, \mathrm{Lip}(\zeta)\cdot 2\Lambda_\omega(\zeta) \cdot 2\mathbf{C}_d \sqrt{\varepsilon t}  = 4\Lambda_3(\zeta)\cdot  \mathrm{Lip}(\zeta) \cdot \Lambda_\omega(\zeta)\cdot\mathbf{C}_d\sqrt{\varepsilon t}. 
\end{align*}
From \eqref{eq:B} we have 
\begin{align*}
    \int_{\R^d} B(x,t) \omega(x)dx \leq d\Lambda_H \cdot  \varepsilon t . 
\end{align*}
Summing $A_1,A_2$ and $B$, we obtain the conclusion. 
\end{proof}

\else \fi


If $x \mapsto H(x,p)$ is $\Z^d$-periodic, we may take $\omega \equiv 1$ on $\T^d$ in Proposition \ref{prop:commutatorL1}, yielding the following result.

\begin{cor}\label{cor:periodic}
    Assume \ref{itm:H1}--\ref{itm:H2} and that $x\mapsto H(x,p)$ is $\Z^d$-periodic. 
    If $\zeta\in \mathrm{Lip}(\R^d)$ is semiconcave then 
\begin{align*}
    -\Lambda_1(\zeta)\cdot \varepsilon t\leq \int_{\T^d} \left( \mathcal{S}^{\mathrm{H}}_t \left(x,D\zeta(x)\right) - H(x,\mathcal{S}^{\mathrm{H}}_t D\zeta(x))\right)dx 
    \leq 
    \Lambda^{\mathrm{BV}}_2(\zeta)\sqrt{\varepsilon t}
\end{align*}
where $\Lambda^{\mathrm{BV}}_2(\zeta)$ depends only on $\mathrm{Lip}(\zeta)$ and the total variation $\Vert D^2\zeta\Vert(\R^d)$. 
\end{cor}

\subsection{Splitting Error in $L^1(\T^d)$ and Proof of Proposition \ref{prop:errorTorusTL}}

\begin{lem}\label{lem:L1compare}
Let $z\in \T^d\times [0,h]$ be a classical solution to the following problem
\begin{equation*}
    \begin{cases}
        \begin{aligned}
            z_t + \mathbf{b}(x,t)\cdot Dz - \varepsilon \Delta z &= \phi(x,t) &&\qquad  (x,t) \in \T^d\times (0,h) \\
            z(x,0) &= z_0(x) &&\qquad  x\in \R^d, 
        \end{aligned}
    \end{cases}
\end{equation*}
where $\phi\in L^1(\T^d)\times (0,h)$, $z_0\in L^1(\T^d)$, and $\mathbf{b}\in L^1((0,h);L^\infty(\T^d))$, such that $\mathrm{div}(\mathbf{b}(x,t)) \leq \ell$ in $\mathcal{D}'(\T^d)$.
Then 
\begin{equation*}
    \Vert z(\cdot, t)\Vert_{L^1(\T^d)} \leq e^{\ell t}
    \left(
        \Vert z_0\Vert_{L^1(\T^d)} + \int_0^t \Vert \phi(\cdot, s)\Vert_{L^1(\T^d)}ds
    \right). 
\end{equation*}
\end{lem}

\ifprintanswers

\begin{proof}[Proof of Lemma \ref{lem:L1compare}] Multiply both sides of the equation by \(\mathrm{sign}(z(x))\). 
A standard regularization of the sign function can be used to justify this step rigorously. We have
\begin{align*}
    - & \int_{\T^n} \mathrm{sign}(z) \Delta z\;dx  
    = 
    \int_{\T^n}  \frac{Dz}{|Dz|}\cdot Dz\;dx = \int_{\T^n} |Dz|\;dx \geq 0 \\
    & \int_{\T^d} \mathrm{sign}(z\;\mathbf{b}(x,t)\cdot Dz(x)\;dx = \int_{\T^n} \mathbf{b}(x,t) \cdot D(|z|)\;dx = -\int_{\T^n} \mathrm{div}(\mathbf{b})|z|\;dx .
\end{align*}
From 
\begin{align*}
    \frac{d}{dt}\int_{\T^n} |z|\;dx 
    + 
    \int_{\T^d} \mathrm{sign}(z)\;\mathbf{b}\cdot Dz\;dx  - \varepsilon\int_{\T^d} \mathrm{sign}(z)\Delta z\;dx = \int_{\T^n} f(x,t)\;\mathrm{sign}(z)\;dx
\end{align*}
we have
\begin{equation*}
     \frac{d}{dt}\int_{\T^n} |z|\;dx 
     \leq \int_{\T^n} |f(x,t)|\;dx + \ell\int_{\T^n} |z|\;dz. 
\end{equation*}
By Gr\"onwall's inequality we obtain the conclusion.
\end{proof}

\else \fi 


\begin{proof}[Proof of Proposition \ref{prop:errorTorusTL}]
    For $i=1,2,\ldots, n$ we define $v(x,t_i)$, $\zeta^{(i)}(x,t)$ as in \eqref{eq:onevRd} with $u_0\equiv v_0$. 
    We have $\|Dv(\cdot,t_i)\|_{L^\infty(\R^d)}
        \leq \|D\zeta^{(i)}(\cdot,h)\|_{L^\infty(\R^d)}
        \leq K_0(u_0)$
    by Proposition \ref{prop:BoundForDelta},
    where $K_0(u_0)$ is given by \eqref{eq:K0}, and $\mathrm{SC}(\zeta^{(i)}(\cdot,h)) \le K_2(u_0)$, 
    with $K_2(u_0)$ from \eqref{eq:K2}.
    Define
     \begin{align*}
        q^{(i)}(x,t) := \mathcal{S}^{\mathrm{H}}_{t} \circ \mathcal{S}^{\mathrm{HJ}}_{t} v(\cdot,t_{i-1})(x) = \mathcal{S}^{\mathrm{H}}_{t} \zeta^{(i)}(x,t). 
    \end{align*}
    Using Lemma \ref{lem:diffS} we have 
    \begin{align*}
    & \partial_t q^{(i)} + H(x,Dq^{(i)}) - \varepsilon\Delta q^{(i)} 
    = 
    H(x,D\mathcal{S}^{\mathrm{H}}_t\zeta^{(i)}(x,t)) - \mathcal{S}_t^{\mathrm{H}} H(x,D\zeta^{(i)}(x,t)) = R(x,t). 
    \end{align*}
    The lower bound follows as in Proposition \ref{prop:key-result-lower}, since the $L^1(\T^d)$ estimate is implied by the $L^\infty(\T^d)$ bound. For the upper bound, Corollary \ref{cor:periodic} implies that
    \begin{equation*}
        \Vert R(\cdot, t)\Vert_{L^1(\T^n)} \leq \Lambda_2^{\mathrm{BV}}(u_0)\sqrt{\varepsilon t}
    \end{equation*}
    where $\Lambda_2^{\mathrm{BV}}(u_0)$ depends only on $H$, $d$, $\mathrm{Lip}(u_0)$ and $\mathrm{SC}(u_0)$. 
   Let 
   \begin{equation*}
    z(x,t) = u(x,t_{i-1}+t) - q^{(i)}(x,t) \qquad\text{for}\; (x,t) \in \T^d\times [0,h].     
   \end{equation*}
    Then $z$ is a classical solution to 
    \begin{align*}
        z_t + H(x,Du(x,t_{i-1}+t)) - H(x,Dq^{(i)}(x,t))  - \varepsilon \Delta z = -R(x,t) \qquad\text{in}\;\T^d\times (0,h). 
    \end{align*}
    For $(x,t) \in \T^d\times (0,h)$, let
    \begin{equation*}
        \mathbf{b}(x,t) = \int_0^1 D_pH\big(x,sDu(x,t_{i-1}+t) + (1-s)Dq^{(i)}(x,t)\big)\;ds. 
    \end{equation*}
We note that $\mathbf{b}$ is defined everywhere, since $u$ and $q^{(i)}$ are both $C^2$. We have
\begin{align*}
    H(x,Du(x,t_{i-1}+t)) - H(x,Dq^{(i)}(x,t)) = \mathbf{b}(x,t)\cdot Dz(x,t) . 
\end{align*}
It is clear that $\mathbf{b}\in L^\infty(\T^d\times [0,h])$. 
We compute  
\begin{align*}
    \mathrm{div}\,\mathbf{b}(x,t)
        &= \int_0^1 \mathrm{tr}\big(D^2_{xp}H(x,p_s)\big)\,ds \\
        &\quad + \int_0^1 \mathrm{tr}\Big(D^2_{pp}H(x,p_s)
        \big(sD^2u(x,t_{i-1}+t) + (1-s)D^2q^{(i)}(x,t)\big)\Big)\,ds,
\end{align*}
where $p_s := sDu(x,t_{i-1}+t) + (1-s)Dq^{(i)}(x,t)$. 
Using $D^2u,\, D^2q^{(i)} \preceq K_2(u_0)$, $D_{pp}H \geq 0$, and the Lipschitz bounds on $Du$ and $Dq^{(i)}$, \ref{itm:H2} yields
\begin{align*}
    \mathrm{div}(\mathbf{b}(x,t)) \leq  C \qquad\text{in}\;\mathcal{D}'(\T^d)
\end{align*}
where $C = C(u_0)$ depends on $H, \mathrm{Lip}(u_0)$ and $\mathrm{SC}(u_0)$. By Lemma \ref{lem:L1compare} we have 
\begin{align*}
    \Vert z(\cdot, t)\Vert_{L^1(\T^d)} 
    &\leq e^{Ct} \left(
        \Vert z(\cdot, 0)\Vert_{L^1(\T^d)} 
        + 
        \int_0^t \Vert R(\cdot,t)\Vert_{L^1(\T^d)}\;ds
    \right).
\end{align*}
In other words, at $t=h$ we have
\begin{align*}
    \Vert u(\cdot, t_{i}) - v(\cdot, t_i) \Vert_{L^\infty(\T^d)}
    &\leq 
    e^{Ch} \left(
        \Vert u(\cdot, t_{i-1}) - v(\cdot, t_{i-1}) \Vert_{L^\infty(\T^d)}
        + 
        \Lambda_2^{\mathrm{BV}}(u_0)\sqrt{\varepsilon} h^{3/2}
    \right) \\
    &\leq 
        \Lambda_2^{\mathrm{BV}}(u_0) (e^{CT}-1)\sqrt{\varepsilon h}  .
\end{align*}
The conclusion follows.      
\end{proof}


\section{First-order Hamilton-Jacobi equations: PI-$\lambda$ Algorithm}
\label{sec:FirstOrder}
In this section, we study the first-order Hamilton–Jacobi equation \eqref{eq:HJperator}, whose solution is given by the value function of the optimal control problem \eqref{eq:eta}--\eqref{eq:valueuHJB}; see Subsection~\ref{subsubsection:HJB1st}.

\subsection{Value-gradient policy iteration--based algorithm}
We refer to \cite{puterman_convergence_1979} 
and the references therein for the standard policy 
iteration (Algorithm \ref{alg:policy-iteration}) and its convergence analysis.
\begin{algorithm}[htbp]
\caption{Standard Policy iteration}\label{alg:policy-iteration}
\begin{enumerate}
    \item[Step 1.] Given $u^{(k)}(x,t)$, we solve for $a^{(k+1)}(x,t)$:
    \begin{align}\label{eq:u-PDE}
    \begin{cases}
    \begin{aligned}
        \partial_t u^{(k)} - Du^{(k)}\cdot f(x,a^{(k)}) - \ell(x,a^{(k)}) &= 0 &&\text{in}\;\R^d\times (0,T), \\
        u^{(k)}(x,0) &= u_0(x) &&\text{on}\;\R^d. 
    \end{aligned}
    \end{cases}
    \end{align}

    \item[Step 2.] Policy update: $a^{(k+1)}$ is obtained from the optimization problem 
    \begin{align*}
        a^{(k+1)}(x,t) = \arg\max_{a\in \mathrm{A}} \Big( Du^{(k)}(x,t)\cdot \big( -f(x,a) \big) - \ell(x,a) \Big). 
    \end{align*}
\end{enumerate}
\end{algorithm}

We recall the value-gradient policy iteration–based 
Algorithm \ref{alg:PI-lambda-time} and 
derive \eqref{eq:PDEforLambda}. 
Let $\lambda(x,t) = D u(x,t)$ for $(x,t) \in \R^d\times (0,T)$. 
Let $\hat{a}(\cdot)$ be the optimal policy in 
Algorithm \ref{alg:policy-iteration}, we have 
\begin{align}
    & \partial_t u(x,t) = \lambda(x,t)\cdot f(x,\hat{a}(x,t)) + \ell(x,\hat{a}(x,t)), \qquad \label{eq:HJ1stOrderTimeLambda} \\
    & \hat{a}(x,t) = \arg\max_{a\in \R^p} \Big( Du(x,t) \cdot \big(-f(x,a)\big) - \ell(x,a) \Big), 
    \label{eq:OptimalPolicty}
\end{align}
Let us derive the PDE for $\lambda(x,t)$ based on \eqref{eq:HJ1stOrderTimeLambda}. 
We write \(\lambda = (\lambda_1,\ldots,\lambda_d)\), \(f = (f_1,\ldots,f_d)\), and \(a = (a_1,\ldots,a_p)\).
Taking the derivative with respect to $x_i$ for $i\in \{1,\ldots, n\}$ of \eqref{eq:HJ1stOrderTimeLambda}, we obtain
\begin{align*}
    \frac{d}{dt}\cdot \frac{\partial u}{\partial x_i} 
    = \sum_{j=1}^d
    \left[
        \frac{\partial \lambda_j}{\partial x_i}\cdot f_j\big(x,\hat{a}\big)
        +
        \lambda_j\cdot \frac{\partial f_j}{\partial x_i}
    \right] 
        + \frac{\partial \ell}{\partial x_i} 
    + 
    \sum_{k=1}^p \frac{\partial \hat{a}_k}{\partial x_i}
    \left[
        \sum_{j=1}^n \lambda_j \cdot \frac{\partial f_j}{\partial a_k} 
        + 
        \frac{\partial \ell}{\partial a_k}
    \right]. 
\end{align*}
Here $\hat{a} = (\hat{a}_1, \ldots, \hat{a}_p)$, and $\hat{a}_k=\hat{a}_k(x,t)$ for each component $k=1,2,\ldots, p$. 
Assume that \eqref{eq:OptimalPolicty} 
admits a unique optimal solution $\hat{a}(x,t)$. 
Then the following first-order necessary condition holds:
\begin{align}\label{eq:1stCancel}
    \sum_{j=1}^n \lambda_j(x,t) \frac{\partial f_j}{\partial a_k}(x,\hat{a}) + \frac{\partial \ell}{\partial a_k}(x,\hat{a}) = 0, \qquad k=1,2,\ldots, p. 
\end{align}
We deduce that 
\begin{align*}
    \frac{d}{dt}\cdot \frac{\partial u}{\partial x_i}(x,t) 
    = 
    \sum_{j=1}^n
    \left[
        \frac{\partial \lambda_j}{\partial x_i}(x,t)\cdot f_j\big(x,\hat{a}(x,t)\big)
        +
        \lambda_j(x,t)\cdot \frac{\partial f_j}{\partial x_i}
    \right] + \frac{\partial \ell}{\partial x_i}
\end{align*}
for $i=1,2,\ldots, d$, 
where $\hat{a}$ satisfies \eqref{eq:OptimalPolicty}. 
In other words, assuming $D\lambda$ is symmetric, 
\begin{align}\label{eq:CompactFormLambda}
\begin{cases}
\begin{aligned}
    & \lambda_t = D^T\lambda\cdot f(x,\hat{a}(x,t)) + D ^Tf(x,\hat{a}(x,t))\cdot \lambda + \nabla_x \ell(x,\hat{a}(x)), \\
    & \lambda(x,0) = \nabla u_0(x) .
\end{aligned}
\end{cases}
\end{align}

\subsection{Key Estimates and Error Analysis}

\begin{prop}
  \label{prop:estimate-lambda}
Assume \ref{itm:A0}, \ref{itm:A1}. We define $ \lambda^{(0)}(x,t) = \nabla u_0(x)$. 
    For every $k\in \N$ and $(x,t) \in \R^d\times (0,T)$ we have
    \begin{align}\label{eq:ak1}
        |a^{(k)}(x,t)| \leq \frac{C_f}{\ell_0}|\lambda^{(k)}(x,t)| + 1. 
    \end{align}
    %
    Furthermore, if 
    \begin{align}\label{eq:TsmallA}
    T\leq T_1 := \left\lbrace 
    \frac{\ln\left(\frac{3}{2}\right)}{2C_f(1+ 2C_0^2\frac{C_f}{\ell_0})}, 
    \frac{\frac{C_0}{3}}{C_f + C_\ell\left(1 + 2C_0\frac{C_f}{\ell_0}\right)}
    \right\rbrace 
    \end{align}
    then $|\lambda^{(k)}(x,t)| \leq 2C_0(1+|x|)$ for all $(x,t)\in \R^d\times (0,T)$. 
\end{prop}

\ifprintanswers


\begin{proof}[Proof of Proposition \ref{prop:estimate-lambda}] 
The estimate for $a^{(k)}$ follows from \eqref{eq:akmin} and the structure of $f$ in \ref{itm:A1}. We note that $c_f^T\cdot \lambda^{(k)}(x,t) + D_a\ell\big(x,a^{(k)}(x,t)\big) = 0$, therefore
\begin{equation}\label{eq:ak}
    \ell_0(|a^{(k)}(x,t)| - 1) \leq C_f|\lambda^{(k)}(x,t)| 
\end{equation}
by using the lower bound of $D_a\ell$ in \ref{itm:A1}. Since 
\begin{align*}
    a^{(k)}(x,t) = \arg\min_a 
		    \left(  
        		f(x, a)\cdot \lambda^{(k)}(x,t) +  \ell(x,a)
		    \right)     
\end{align*}
we deduce the conclusion \eqref{eq:ak1}. \medskip 
            
We estimate $\lambda^{(k)}$ using induction.
By \ref{itm:A0}, $|\lambda^{(0)}(x,t)| \leq \Lambda_0(1+|x|)$. 
We derive a constant $\Theta_k$ such that
\begin{align}\label{eq:assumptionk}
    |\lambda^{(k)}(x,t)| \leq \Theta_k(1+|x|) \qquad\Longrightarrow \qquad |\lambda^{(k)}(x,t)| \leq \Theta_{k+1}(1+|x|),
\end{align}
where $\Theta_0 = \Lambda_0$. We show that if $T$ is sufficiently small, then $\Theta_k \le 2\Lambda_0$ implies $\Theta_{k+1}\leq 2\Lambda_0$ as well. 
%
For simplicity, by \eqref{eq:ak1} we assume
    \begin{equation}\label{eq:assumptionkak}
        |a^{(k)}(x,t)| \leq A_k(1+|x|), \qquad \text{where}\qquad A_k = \frac{C_f}{\ell_0} \Theta_k + 1. 
    \end{equation}
    Let $(x,t)\in \R^d\times (0,T)$ and $\eta:[0,t]\to \R^d$ be a trajectory to 
\begin{equation}\label{eq:cha}
	\begin{cases}
	\begin{aligned}
		\dot{\eta}(s) &= -f\big(\eta(s), a^{(k)}(\eta(s), s)\big) , \qquad s\in (0,t) \\
			\eta(t) &= x. 
	\end{aligned}
	\end{cases}
\end{equation}
From the growth of $f$ in \ref{itm:A1} we have
\begin{align*}
    \big|f\big(\eta(s), a^{(k)}(\eta(s),s)\big)\big| 
        \leq 
    C_f(1+|\eta(s)| + |a^{(k)}(\eta(s),s)|) \leq C_f(1+A_k)(1+|\eta(s)|). 
\end{align*}
Since $\eta(t)=x$, integrating backward from $t$ to $s$ gives
\begin{align*}
    1 + |\eta(s)| 
    \leq (1+|x|) + C_f(1+A_k)\int_s^t (1+|\eta(r)|)\,dr.
\end{align*}
By Gr\"onwall's inequality (backward in time), we obtain
\begin{align}\label{eq:boundeta0}
    1+|\eta(s)| \leq (1+|x|)\,e^{C_f(1+A_k)(t-s)}, \; s\in [0,t].
\end{align}
Therefore
\begin{align*}
    1+|\eta(0)| \leq (1+|x|)\,e^{C_f(1+A_k)t}.     
\end{align*}
Along the characteristic \eqref{eq:cha}, equation \eqref{eq:lambda-k} yields
\begin{align*}
      & \frac{d}{ds}\left( \lambda^{(k+1)}(\eta(s),s)\right) \\
      &\qquad = D_x f(\eta(s),a^{(k)}(\eta(s),s))\cdot \lambda^{(k)}(\eta(s),s) + D_x\ell(\eta(s),a^{(k)}(\eta(s),s)). 
\end{align*}
From the growth conditions of $Df, \ell$ in \ref{itm:A1} and \eqref{eq:assumptionkak} we have
\begin{align*}
    \left|\frac{d}{ds}\left( \lambda^{(k+1)}(\eta(s),s)\right)\right| 
    \leq (C_f + C_\ell(1+A_k))(1+|\eta(s)|), \qquad s\in [0,t]. 
\end{align*}
Integrating from $0$ to $t$ and using the initial condition 
$\lambda^{(k+1)}(\eta(0),0) = \nabla u_0(\eta(0))$, we obtain
\begin{align*}
    |\lambda^{(k+1)}(\eta(t),t)|
    &\leq \Lambda_0(1+|\eta(0)|) +  (C_f + C_\ell(1+A_k)) \int_{0}^t (1+|\eta(s)|)\,ds,
\end{align*}
where we used \ref{itm:A0} that $|\nabla u_0(x)|\leq \Lambda_0(1+|x|)$. 
Since $\eta(t) = x$, using \eqref{eq:boundeta0} we obtain
\begin{align*}
     &|\lambda^{(k+1)}(x,t)|  \\
     &\leq \Lambda_0(1+|x|) e^{C_f(1+A_k)t} +  (C_f + C_\ell(1+A_k)) (1+|x|)\int_0^t e^{C_f(1+A_k)(t-s)}\,ds \\
     &\leq (1+|x|)\left[\Lambda_0 \cdot e^{C_f(1+A_k)T} + \frac{C_f+C_\ell(1+A_k)}{C_f(1+A_k)}\left(e^{C_f(1+A_k)T} - 1\right)\right] \\
     &\leq (1+|x|) \left[
     \Lambda_0\cdot  e^{C_f(1+A_k)T} + \frac{C_f+C_\ell(1+A_k)}{C_f(1+A_k)}\cdot C_f(1+A_k)T\cdot e^{C_f(1+A_k)T} 
     \right] \\
     & = (1+|x|) e^{C_f(1+A_k)T} \left[
     \Lambda_0 + 
     \big(C_f + C_\ell(1+A_k)\big) T
     \right]
\end{align*}
where we use $e^{x} - 1 \leq xe^x$ for all $x\geq 0$. Choose $T$ satisfying \eqref{eq:TsmallA} such that
\begin{align*}
    e^{C_f(1+A_k)T} \leq \frac{3}{2} \qquad\text{and}\qquad \Lambda_0 + 
     \big(C_f + C_\ell(1+A_k)\big) T \leq \frac{4}{3}\Lambda_0, 
\end{align*}
then $\Theta_{k+1}\leq \frac{3}{2}\cdot\frac{4}{3}\Lambda_0 = 2\Lambda_0$. 
We obtain the conclusion by induction. 
\end{proof}

\else
\fi


\begin{prop}\label{prop:BoundDaDlambda} 
  Assume \ref{itm:A0}, \ref{itm:A1}, \ref{itm:A2}. 
  We define $ \lambda^{(0)}(x,t) = \nabla u_0(x)$. 
    For every $k\in \N$ we have
    \begin{align}\label{eq:Dak}
        |D_xa^{(k)}(x,t)| 
        \leq \frac{1}{\ell_0}
        \left(C_f |D\lambda^{(k)}(x,t)| + C_\ell \right). 
    \end{align}
    Furthermore, let 
    \begin{align}\label{eq:Tvalue}
    T_2 := \min 
    \left\lbrace 
        \frac{\ln\left(\frac{3}{2}\right)}{C_f\left(1+ \frac{2C_0C_f+C_\ell}{\ell_0}\right)}, 
        \frac{\frac{1}{3}C_0}{2C_0C_f\left(1+\frac{C_\ell}{\ell_0}\right) + 2C_fC_0 + C_\ell + \frac{C_\ell^2}{\ell_0}}
    \right\rbrace 
    \end{align}
     then for $T\leq T_2$ and $(x,t)\in \R^d\times (0,T)$ we have
    \begin{align}\label{eq:boundDlambdak}
        |D\lambda^{(k)}(x,t)| \leq 2C_0 . 
    \end{align}
    %
    Lastly, for $k\in \N$ and $(x,t)\in \R^d\times [0,T]$ we have
    \begin{align}\label{eq:est-a-lambda}
        |a^{(k)}(x,t) -  a^{(k-1)}(x,t)| \leq \frac{C_f}{\ell_0} |\lambda^{(k)}(x,t) - \lambda^{(k-1)}(x,t)|. 
    \end{align}
\end{prop}

\ifprintanswers


\begin{proof}[Proof of Proposition \ref{prop:BoundDaDlambda}] 
    We estimate $Da^{(k)}$ by differentiating equation \eqref{eq:ak} with respect to $x$:
    \begin{align*}
        D^2_{ax} \ell(x,a^{(k)}) + D^2_{aa}\ell(x,a^{(k)})\cdot D_xa^{(k)} + \mathbf{c}_f^T\cdot D\lambda^{(k)} = 0. 
    \end{align*}
Solving for $D_xa^{(k)}(x,t)$, we obtain
\begin{align*}
    D_xa^{(k)}(x,t) = -\left(D^2_{aa}\ell(x,a^{(k)}(x,t))\right)^{-1} 
    \Big( \mathbf{c}_f\cdot D\lambda^{(k)}(x,t) + D^2_{ax} \ell\big(x,a^{(k)}(x,t)\big) \Big).
\end{align*}
By \ref{itm:A1}, $|\mathbf{c}_f| \le C_f$. By \ref{itm:A2}, 
$D^2_{aa}\ell(x,a) \succ \ell_0\,\mathbb{I}_d$ and $|D^2_{ax}\ell(x,a)| \le C_\ell$, 
yielding \eqref{eq:Dak}.
\medskip 

Next, we begin with the base case $k=0$. 
Since $\lambda^{(0)}(x,t)=\nabla u_0(x)$, it follows from \ref{itm:A0} that $|D\lambda^{(0)}(x,t)| = |D^2 u_0(x)| \leq \Lambda_0$. 
We use induction to show that $|D\lambda^{(k)}(x,t)|\leq 2\Lambda_0$ for all $k\in \N$ provided that $T$ is small so that \eqref{eq:Tvalue} is satisfied. Assume 
\begin{align}\label{eq:assumeDlambdak}
    |D\lambda^{(k)}(x,t)| \leq 2\Lambda_0. 
\end{align}
Let $\eta$ be the characteristic curve of \eqref{eq:cha}, and set
$p(s) = D\lambda^{(k+1)}(\eta(s),s)$ for $s\in (t,T)$.
Then
\begin{align}\label{eq:dp}
    \frac{d}{ds}p(s) = \frac{d}{ds}\left(D\lambda^{(k+1)}(\eta(s),s)\right) - D^2\lambda^{(k+1)}(\eta(s),s)\cdot f\left(\eta(s),a^{(k)}(\eta(s),s)\right). 
\end{align}
We derive an equation for $p(s)$. From \eqref{eq:lambda-k} we have
\begin{align*}
    \partial_t \lambda^{(k+1)}(x,t) &= D\lambda^{(k+1)}\cdot f(x,a^{(k)}) + D _xf(x,a^{(k)})\cdot \lambda^{(k)} + \nabla_x \ell(x,a^{(k)}). 
\end{align*}
Taking the derivative in $x$ on both sides, we have
\begin{align*}
    &\partial_t D\lambda^{(k+1)}(x,t) \\
    &\qquad\qquad = D^2 \lambda^{(k+1)}\cdot f(x,a^{(k)}) + D\lambda^{(k+1)}\cdot D_x (f(x,a^{(k)})) + D_x\big(\nabla_x \ell(x,a^{(k)})\big). 
\end{align*}
We have
\begin{align*}
        &D_x \Big(  D\lambda^{(k+1)}\cdot f(x,a^{(k)}) \Big) 
        = 
        D^2\lambda^{(k+1)} \cdot f(x,a^{(k)}) + D\lambda^{(k+1)}\cdot D_x\big(f(x,a^{(k)})\big) \nonumber \\
        &\quad 
        = D^2\lambda^{(k+1)} \cdot f(x,a^{(k)}) + D\lambda^{(k+1)}\cdot D_xf(x,a^{(k)}) 
        +
        D\lambda^{(k+1)}\cdot D_a f(x,a^{(k)})\cdot D_xa^{(k)} \nonumber     \\
        &D_x\Big(D_xf(x,a^{(k)})\cdot\lambda ^{(k)}\Big) 
        = D_x(D_xf(x,a^{(k)}))\cdot\lambda^{(k)} + D_xf(x,a^{(k)})\cdot D\lambda^{(k)} \\
        &D_x\nabla \ell(x,a^{(k)}) 
        = D^2_{xx} \ell(x,a^{(k)}) + D^2_{xa}\ell(x,a^{(k)})\cdot D_x a^{(k)} . 
\end{align*}
From the structure of $f$ in \ref{itm:A1}, we obtain \begin{equation}\label{eq:j1}
    \partial_t D\lambda^{(k+1)} + D^2\lambda^{(k+1)}\cdot f(x,a^{(k)}) 
    =  
    D\lambda ^{(k+1)} \cdot \mathbf{A}^{(k)}(x,t) + \mathbf{B}^{(k)}(x,t)
\end{equation}
where 
\begin{align*}
    \mathbf{A}^{(k)}(x,t) &=  D_xf_0(x)+\mathbf{c}_f\cdot D_xa^{(k)} \\
    \mathbf{B}^{(k)}(x,t) &= D_x(D_xf_0(x))\cdot \lambda^{(k)}+ D_xf_0(x)\cdot D\lambda^{(k)} \\
    &\qquad + D^2_{xx}\ell(x,a^{(k)}) + D^2_{xa}\ell(x,a^{(k)})\cdot D_xa^{(k)}. 
\end{align*}
For $\mathbf{A}^{(k)}$, by the growth of $f$, $|D_x f_0|\le C_f$ in \ref{itm:A1}, and \eqref{eq:Dak}, we obtain
\begin{equation}\label{eq:Ak}
    |\mathbf{A}^{(k)}(x,t)| \leq C_f + \frac{C_f}{\ell_0}\left(C_f|D\lambda^{(k)}(x,t)| + C_\ell\right). 
\end{equation}
For $\mathbf{B}^{(k)}$, we use \eqref{eq:2ndDerivativeDf} in \ref{itm:A2}, the growth of $f$, and the properties of $\ell$ in \ref{itm:A2}, together with $|\lambda^{(k)}(x,t)| \leq 2\Lambda_0 (1+|x|)$, $ |D\lambda^{(k)}(x,t)| \leq 2\Lambda_0$ from Proposition \ref{prop:estimate-lambda} to deduce that 
\begin{align}
    |\mathbf{B}^{(k)}(x,t)| 
        &\leq \frac{C_f}{1+|x|} \cdot 2\Lambda_0(1+|x|) + C_f |D\lambda^{(k)}(x,t)| \\
        &\qquad\qquad  + C_\ell + \frac{C_\ell}{\ell_0}\left(C_f|D\lambda^{(k)}(x,t)| + C_\ell\right) \nonumber \\
        &= C_f\left(1+\frac{C_\ell}{\ell_0}\right)|D\lambda^{(k)}(x,t)| + \left(2C_f\Lambda_0 + C_\ell + \frac{C_\ell^2}{\ell_0}\right).  \label{eq:Bk}
\end{align}
From \eqref{eq:Ak}, \eqref{eq:Bk} and \eqref{eq:assumeDlambdak}, we have
\begin{align*}
    & \left| 
        D\lambda^{(k+1)} \cdot \mathbf{A}^{(k)}(x,t) + \mathbf{B}^{(k)}(x,t)
    \right| 
    \leq K_1 |D\lambda^{(k+1)}(x,t)| + K_2, 
\end{align*}
where 
\begin{align*}
    K_1 &= C_f\left(1+ \frac{2\Lambda_0C_f+C_\ell}{\ell_0}\right), \\ 
    K_2 &= 2\Lambda_0C_f\left(1+\frac{C_\ell}{\ell_0}\right) + 2C_f\Lambda_0 + C_\ell + \frac{C_\ell^2}{\ell_0}. 
\end{align*}
Evaluating along the characteristic using \eqref{eq:dp}, we deduce that 
\begin{align*}
    \left|\dot{p}(s)\right| \leq K_1 |p(s)| + K_2, \qquad s\in (0,t).
\end{align*}
Recall that $|p(0)| = |D\lambda^{(k+1)}(x,0)| = |D^2u_0(x)| \leq \Lambda_0$. Hence, by \ref{itm:A2} we have
\begin{align*}
    |p(t)| \leq |p(0)| + \int_0^t |\dot{p}(s)|\;ds \leq \Lambda_0 + K_2t + K_1\int_{0}^t |p(s)|\;ds. 
\end{align*}
By Gr\"onwall inequality we obtain
\begin{align*}
    |p(t)| \leq \left(\Lambda_0 + K_2(T-t)\right) e^{K_1(T-t)} \leq (\Lambda_0+K_2T) e^{K_1T}.
\end{align*}
Choose $T>0$ sufficiently small such that \eqref{eq:Tvalue} holds, then the conclusion follows: 
\begin{align*}
    |D\lambda^{(k+1)}(x,t)| \leq (\Lambda_0+K_2T )e^{K_1 T} \leq \frac{4}{3}\Lambda_0 \cdot \frac{3}{2} = 2\Lambda_0.
\end{align*}
\medskip 

Lastly, from \eqref{eq:akmin} we have
\begin{align*}
    D_a\ell(x,a^{(k)}(x,t)) + \mathbf{c}_f\cdot \lambda^{(k)}(x,t) = 
    D_a\ell(x,a^{(k-1)}(x,t)) + \mathbf{c}_f\cdot \lambda^{(k-1)}(x,t) = 0
\end{align*}
Therefore
\begin{align*}
    \textbf{c}_f\cdot \left(\lambda^{(k)}(x,t) - \lambda^{(k-1)}(x,t)\right) =  - 
    \left( 
        D_a\ell(x,a^{(k)}(x,t)) -  D_a\ell(x,a^{(k-1)}(x,t)) 
    \right).
\end{align*}
Using the convexity of $\ell(x,a)$ in \ref{itm:A1} we deduce that $\Vert D^2_{aa} \ell(\cdot,\cdot)\Vert \geq \ell_0$. Therefore, the above equation implies \eqref{eq:est-a-lambda}. 
\end{proof}

\else \fi 

We derive an error estimate for Algorithm \ref{alg:PI-lambda-time} as \(k\to\infty\) in a weighted \(L^2\) space--time norm. 
The weight includes a factor \(\gamma\), providing greater flexibility than in the time-independent case, where it acts as the discount factor \cite{bensoussan_value-gradient_2023}.


\begin{proof}[Proof of Theorem \ref{thm:mainthm2}] 
Let \(T \leq T_0 := \min\{T_1, T_2\}\), 
where \(T_1\) and \(T_2\) are given by 
\eqref{eq:TsmallA} and \eqref{eq:Tvalue}, respectively.
We also define
\begin{align*}
        \gamma_0: = 
         2C_f\left(1+\frac{2C_fC_0+C_\ell}{\ell_0} + 5\alpha\right) + 6\left(\frac{2C_0C_f^2 + C_\ell C_f + C_f\ell_0}{2\ell_0}\right). 
\end{align*}
By definition of $\lambda^{(k)}$ we have
\begin{align*}
    \partial_t \lambda^{(k+1)} &=  D \lambda^{(k+1)}\cdot f\big(x, a^{(k)}\big)  
    + 
    D^Tf\big(x,a^{(k)}\big)\cdot \lambda^{(k)}
    + \nabla_x \ell\big(x,a^{(k)}\big) \\
    \partial_t \lambda^{(k)} &=  D \lambda^{(k)}\cdot f\big(x, a^{(k-1)}\big)  
    + 
    D^Tf\big(x,a^{(k-1)}\big)\cdot \lambda^{(k-1)}
    + \nabla_x \ell\big(x,a^{(k-1)}\big).
\end{align*}
Subtracting the two equations and multiply both sides by $e^{-\gamma t}\left(\frac{\lambda^{(k+1)} - \lambda^{(k)}}{(1+|x|^2)^{2\alpha}}\right)$
and integrate, we have
\begin{align*}
    I_0 = \int_0^T \int _{\R^d} e^{-\gamma t} \, \partial_t(\lambda^{(k+1)} - \lambda^{(k)}) \left( \frac{\lambda^{(k+1)} - \lambda^{(k)}}{(1+|x|^2)^{2\alpha}} \right)\;dxdt = I_1 + I_2 +I_3 + I_4 + I_5, 
\end{align*}
where
\begin{align*}
    I_1 &= \int_0^T \int _{\R^d} 
    \big(D\lambda^{(k+1)} - D\lambda^{(k)}\big) f(x,a^{(k-1)}) e^{-\gamma t}\left(\frac{\lambda^{(k+1)} - \lambda^{(k)}}{(1+|x|^2)^{2\alpha}}\right) \;dxdt\\ 
    I_2 &=  \int_0^T \int _{\R^d} D\lambda^{(k+1)}\Big(f(x,a^{(k)}) - f(x,a^{(k-1)})\Big) e^{-\gamma t}\left(\frac{\lambda^{(k+1)} - \lambda^{(k)}}{(1+|x|^2)^{2\alpha}}\right) \;dxdt \\ 
    I_3 &=  \int_0^T \int _{\R^d} 
    D^Tf(x,a^{(k)})\cdot \Big(\lambda^{(k)} - \lambda^{(k-1)}\Big) 
    e^{-\gamma t}\left(\frac{\lambda^{(k+1)} - \lambda^{(k)}}{(1+|x|^2)^{2\alpha}}\right)\;dxdt  \\
    I_4 &= \int_0^T \int _{\R^d} 
     \lambda^{(k-1)}\Big(D^T f(x,a^{(k)}) - D^Tf(x,a^{(k-1)}) \Big)  e^{-\gamma t}\left(\frac{\lambda^{(k+1)} - \lambda^{(k)}}{(1+|x|^2)^{2\alpha}}\right) \;dxdt\\
     I_5 &= 
     \int_0^T e^{-\gamma t}\int _{\R^d} 
     \nabla _x \ell(x, a^{(k)}) - \nabla _x \ell(x, a^{(k-1)})\left(\frac{\lambda^{(k+1)} - \lambda^{(k)}}{(1+|x|^2)^{2\alpha}}\right) \;dxdt. 
\end{align*}
Combining the estimates for $I_i$, 
$i=0,\ldots,5$, from Lemma \ref{lem:EstimatesI0I1I2I3I4I5}, we obtain
\begin{align}\label{eq:ek1error}
    \frac{\gamma}{2} e_{k+1} \leq A e_{k+1} + B e_k
\end{align}
where 
\begin{align*}
    A &= C_f\left(1+\frac{2C_f
    C_0+C_\ell}{\ell_0} + 5\alpha\right) + \frac{2C_0C_f^2 + C_\ell C_f + C_f\ell_0}{2\ell_0} \\
    B &= \frac{2C_0C_f^2 + C_\ell C_f + C_f\ell_0}{2\ell_0}. 
\end{align*}
If $\gamma\geq \gamma_0 = 2A+4B$, then \eqref{eq:ek1error} implies that $e_{k+1} \leq \frac{1}{2}e_k$. 
The proof is complete.
\end{proof}

\begin{lem}\label{lem:EstimatesI0I1I2I3I4I5}
    Under the same assumptions as in Theorem \ref{thm:mainthm2}, 
    we have the following estimates for $I_0, I_1, I_2, I_3, I_4, I_5$:
    \begin{align*}
        I_0 &\geq \frac{\gamma}{2} e_{k+1}, \\
        |I_1| 
        &\leq C_f\left(1+\frac{2C_fC_0+C_\ell}{\ell_0} + 5\alpha\right)e_{k+1}, 
        \qquad
        |I_2| \leq \frac{C_0C_f^2}{\ell_0} (e_{k+1} + e_k), \\
        |I_3| &\leq \frac{C_f}{2}(e_k+e_{k+1}), \qquad 
        I_4 = 0, \qquad
        |I_5| \leq \frac{C_\ell C_f}{2\ell_0}(e_{k}+e_{k+1}).
    \end{align*}
\end{lem}

\ifprintanswers

\begin{proof}[Proof of Lemma \ref{lem:EstimatesI0I1I2I3I4I5}]
%
For the left-hand side, we have
\begin{align*}
I_0
	&= \frac{1}{2} 
		\int_{\R^d}  
        \frac{1}{(1+|x|^2)^{2\alpha}}
			\Bigg( 
			\int_0^T \partial_t \big(|\lambda^{(k+1)} - \lambda^{(k)}|^2\big)
				e^{-\gamma t}\;dt
			\Bigg) dx . 
\end{align*}
For each $x\in \R^d$, integrating by parts in time over $[0,T]$ and using $\lambda^{(k)}(x,0)=\nabla u_0(x)$ for all $k\in\N$, we obtain
\begin{align*}
	\int_0^T \partial_t 
		& \left(|\lambda^{(k+1)} - \lambda^{(k)}|^2\right)e^{- \gamma t}\;dt \\ 
		&\qquad  = e^{\gamma T}|\lambda^{(k+1)}(x,T) - \lambda^{(k)}(x,T)|^2 
    + \gamma\int_0^T e^{-\gamma t}|\lambda^{(k+1)} - \lambda^{(k)}|^2\;dt. 
\end{align*}
As a consequence, we have
\begin{align*}
	I_0  
	\geq 
	\frac{\gamma}{2} \int_0^T e^{-\gamma t}\int_{\R^d} \frac{|\lambda^{(k+1)} - \lambda^{(k)}|^2}{(1+|x|^2)^{2\alpha}}\;dxdt =\frac{\gamma}{2}e_{k+1}. 	
\end{align*}
%
We will make use $|D_xf(x,a)| \leq C_f$, $|D_af(x,a)| = |\mathbf{c}_f|\leq C_f$ from \ref {itm:A1}, and $|\lambda^{(k)}(x,t)| \leq 2C_0(1+|x|)$ from Proposition \eqref{prop:estimate-lambda}. 

\medskip 
To estimate $I_1$, we observe that $D_x(\lambda^{(k+1)} - \lambda^{(k)}) = \frac{1}{2}D_x\big(|\lambda^{(k+1)} - \lambda^{(k)}|^2\big)$, which means
\begin{align*}
	I_1 
    &= \frac{1}{2}\int_0^Te^{-\gamma t}\int_{\R^d} D\big(|\lambda^{(k+1)}-\lambda^{(k)}|^2\big) \frac{f(x,a^{(k-1)})}{(1+|x|^2)^{2\alpha}}\;dx\;    dt. 
\end{align*}
By integration by parts, the boundary terms in the following integral vanish due to the linear growth $|\lambda^{(k)}(x,t)| \leq 2C_0(1+|x|)$ and the fact that $\alpha>1$: 
\begin{align}\label{eq:I1main}
	&\int_{\R^d} D_x\left(|\lambda^{(k+1)}-\lambda^{(k)}|^2\right) \cdot \frac{f(x,a^{(k-1)})}{(1+|x|^2)^{2\alpha}}\;dx \nonumber \\
	&\qquad = 
	-\int_{\R^d} \frac{| \lambda^{(k+1)}-\lambda^{(k)}|^2}{(1+|x|^2)^{2\alpha}} 	
	\left(
		\nabla_x \left( f(x,a^{(k-1)}) \right)
		-4
		f(x,a^{(k)})\frac{\alpha x}{(1+|x|^2)}
	\right).
\end{align}
We estimate
\begin{align}\label{eq:I1a}
	\left| D_x \Big( f(x,a^{(k-1)}) \Big) \right| 
    & = 
        \left| D_x f(x,a^{(k-1)}) + D_a f(x,a^{(k-1)})\cdot Da^{(k-1)} \right| \nonumber \\
    & \leq C_f + C_f |Da^{(k-1)}| 
      \leq 
      C_f + C_f \cdot \frac{2C_fC_0 + C_\ell}{\ell_0}, 
\end{align}
where we use the bound for $Da^{(k-1)}$ from \eqref{eq:Dak} and the bound for $D\lambda^{(k)}$ from \eqref{eq:boundDlambdak}. 
On the other hand, we have $|f(x, a^{(k-1)}(x))| \leq C_f(1+|x|)$, thus 
\begin{align}\label{eq:I1b}
		\left| 
		f(x,a^{(k)})\frac{\alpha x}{(1+|x|^2)}
		\right| \leq 
		C_f\frac{\alpha |x|(1+|x|)}{(1+|x|^2)^{2\alpha}} \leq  \frac{5C_f\alpha}{4}. 
\end{align}
Using \eqref{eq:I1a} and \eqref{eq:I1b} in \eqref{eq:I1main} we deduce that
\begin{align*}
    & \int_{\R^d} D_x\left(|\lambda^{(k+1)}-\lambda^{(k)}|^2\right) \cdot \frac{f(x,a^{(k-1)})}{(1+|x|^2)^{2\alpha}}\;dx \\ 
    &\qquad\qquad  \leq C_f\left(1+\frac{2C_fC_0+C_\ell}{\ell_0} + 5\alpha\right) \int_{\R^d}\frac{f(x,a^{(k-1)})}{(1+|x|^2)^{2\alpha}}\;dx .
\end{align*}
Therefore 
\begin{align*}
	|I_1|\leq  C_f\left(1+\frac{2C_fC_0+C_\ell}{\ell_0} + 5\alpha\right)e_{k+1}. 
\end{align*}
To estimate $I_2$, we write
\begin{align*}
	I_2 = \int_0^T e^{-\gamma t} \int_{\R^d} 
		D\lambda^{(k+1)}\cdot \Big(f(x,a^{(k)}) - f(x,a^{(k-1)})\Big) 
		\cdot
		\frac{\lambda^{(k+1)} - \lambda^{(k)}}{(1+|x|^2)^{2\alpha}} 
		dxdt. 		
\end{align*}
From \eqref{eq:boundDlambdak}, $f(x,a) = \textbf{c}_f\cdot a + f_0(x)$, and \eqref{eq:est-a-lambda} we deduce that 
\begin{align*}
	& |I_2| 
	\leq 
    \frac{2C_0C_f^2}{\ell_0} 
    \int_0^T e^{-\gamma t}
		\int_{\R^d}
        |\lambda^{(k)} - \lambda^{(k-1)}| \cdot \frac{|\lambda^{(k+1)} - \lambda^{(k)}|}{(1+|x|^2)^{2\alpha}} \;dx dt\\
	&\leq \frac{C_0C_f^2}{\ell_0} 
	\left(
		\int_0^T  e^{-\gamma t} \int_{\R^d} \frac{|\lambda^{(k+1)} - \lambda^{(k)}|^2}{(1+|x|^2)^{2\alpha}} \;dxdt
		+ 
		\int_0^T e^{-\gamma t} \int_{\R^d} \frac{|\lambda^{(k)} - \lambda^{(k-1)}|^2}{(1+|x|^2)^{2\alpha}} \;dxdt
	\right) 
    \\
    &\leq  \frac{C_0C_f^2}{\ell_0} (e_{k+1} + e_k). 
\end{align*}
For $I_3$, from $|D_xf|\leq C_f$ we have
\begin{align*}
|I_3| 
    &= \left|\int_0^T e^{-\gamma t}
		\int_{\R^d} D_xf(x,a^{(k)})\cdot \big(\lambda^{(k)} - \lambda^{(k-1)}\big)\frac{\big( \lambda^{(k+1)} - \lambda^{(k)} \big)}{(1+|x|^2)^{2\alpha}}\;dxdt \right| \\
    &\leq C_f  \int_0^T e^{-\gamma t}
		\int_{\R^d}
        |\lambda^{(k)} - \lambda^{(k-1)}| \cdot \frac{|\lambda^{(k+1)} - \lambda^{(k)}|}{(1+|x|^2)^{2\alpha}} \;dx dt \leq \frac{C_f}{2}(e_k+e_{k+1}). 
\end{align*}
We observe that $I_4 = 0$. Indeed, by \eqref{eq:formf} we have $f(x,a) = f_0(x) + \mathbf{c}_f \cdot a$, and hence
\begin{align*}
	\lambda^{(k-1)}\Big(D^T f(x,a^{(k)}) - D^Tf(x,a^{(k-1)}) \Big) = 0
\end{align*}
since $D_xf(x,a) = Df_0(x)$. 
For $I_5$, as in the estimate of $I_2$, we have
\begin{align*}
|I_5| 
    &\leq C_\ell \int_0^T e^{-\gamma t}
		\int_{\R^d} \left|a^{(k)} - a^{(k-1)}\right|\frac{\big(\lambda^{(k+1)} - \lambda^{(k)}\big)}{(1+|x|^2)^{2\alpha}}\;dxdt 
    \leq \frac{C_\ell C_f}{2\ell_0}(e_{k}+e_{k+1}). 
\end{align*}

\end{proof}

\else \fi



\section{Machine Learning Methods and Numerical Experiments} 
\label{sec:ML}

\subsection{Machine Learning Characteristic Methods}
Consider the linear equation \eqref{eq:u-PDE}. 
A forward Euler discretization in time yields
$$
    \frac{u^{(k)}(x,s+h)-u^{(k)}(x,s)}{h} 
        - 
    Du^{(k)}(x,s)\cdot f(x,a^{(k)}(x,s))
    =
    \ell(x,a^{(k)}(x,s)).
$$
Setting $\kappa:=\frac{1}{h}$, $ \mathcal V(x):=u^{(k)}(x,s+h)$, 
\begin{align*}
    \mathcal G(x):=f(x,a^{(k)}(x,s)),   \quad 
    \mathcal L(x) :=\ell(x,a^{(k)}(x,s)) +\frac{1}{h}u^{(k)}(x,s), 
\end{align*}
we obtain the stationary transport equation
\begin{align}\label{eq:stationary_transport}
    \kappa \mathcal V(x) - D\mathcal V(x)\cdot \mathcal G(x)
    = \mathcal L(x), \qquad x\in\mathbb{R}^d.
\end{align}
Let $\mathbf{x}(\mathfrak t)$ solve the characteristic system
\begin{align}\label{eq:characteristic}
    \frac{\mathrm d}{\mathrm d\mathfrak t}\mathbf{x}(\mathfrak t)
    =
    \mathcal G(\mathbf{x}(\mathfrak t)),
    \qquad
    \mathbf{x}(0)=\mathbf{x}_0\in\mathbb{R}^d.
\end{align}
Along $\mathbf{x}(\mathfrak t)$, equation \eqref{eq:stationary_transport}
reduces to the ODE: 
$   \kappa \mathcal V(\mathbf{x}(\mathfrak t))
-
\frac{\mathrm d}{\mathrm d\mathfrak t}
\mathcal V(\mathbf{x}(\mathfrak t))
=
\mathcal L(\mathbf{x}(\mathfrak t))$.
Therefore
$$
    \frac{\mathrm d}{\mathrm d\mathfrak t}
\Big(
e^{-\kappa\mathfrak t}
\mathcal V(\mathbf{x}(\mathfrak t))
\Big)
=
-
e^{-\kappa\mathfrak t}
\mathcal L(\mathbf{x}(\mathfrak t))
$$
by multiplying by $e^{-\kappa\mathfrak t}$. Assuming
\(
\lim_{\mathfrak t\to\infty}
e^{-\kappa\mathfrak t}
\mathcal V(\mathbf{x}(\mathfrak t))=0,
\)
we obtain the representation formula
\begin{align}\label{eq:representation_V}
\mathcal V(\mathbf{x}(\mathfrak t))
=
\int_{\mathfrak t}^{\infty}
e^{-\kappa(\mathfrak t'-\mathfrak t)}
\mathcal L(\mathbf{x}(\mathfrak t'))
\,\mathrm d\mathfrak t'.
\end{align}
Differentiating \eqref{eq:stationary_transport} with respect to $x_i$
and applying the same argument gives, for each component
$\lambda_i=\partial_{x_i}\mathcal V$,
\begin{align}\label{eq:representation_lambda}
\lambda_i(\mathbf{x}(\mathfrak t))
=
\int_{\mathfrak t}^{\infty}
e^{-\kappa(\mathfrak t'-\mathfrak t)}
\partial_{x_i}\mathcal L(\mathbf{x}(\mathfrak t'))
\,\mathrm d\mathfrak t'.
\end{align}

For numerical implementation, we generate characteristic curves
$\{\mathbf{x}^{(n)}\}_{n=1}^N$ from initial states
$\{\mathbf{x}_0^{(n)}\}_{n=1}^N$
(e.g., uniformly sampled in the state space).
Equations \eqref{eq:representation_V}–\eqref{eq:representation_lambda}
allow simultaneous evaluation of $\mathcal V$ and
$\nabla\mathcal V$ along each trajectory.
Each characteristic trajectory is discretized into a finite set of approximately equispaced snapshots in space. 
To extend the labelled data from these trajectories to the whole domain, we introduce a parametric approximation 
$\widehat{\mathcal V}_\theta(x)$ (e.g., Galerkin expansions, radial basis functions, or neural networks). 
The parameter vector $\theta$ is determined by minimizing the loss functional
\begin{align}
\label{def:loss-fun}
\begin{aligned}
    L(\theta) &=
    \mu \sum_{n=1}^N
    \int_{0}^{T}
    \big\|
    \mathcal{V}(\mathbf{x}^{(n)}(\mathfrak t))
        -
    \widehat{\mathcal{V}}_{\theta}(\mathbf{x}^{(n)}(\mathfrak t))
    \big\|^2
    \, d\mathfrak t \\
    &\qquad\qquad  +
    (1-\mu)
    \sum_{n=1}^N
    \int_{0}^{T}
    \big\|
    \lambda(\mathbf{x}^{(n)}(\mathfrak t))
    -
    \nabla \widehat{\mathcal{V}}_{\theta}(\mathbf{x}^{(n)}(\mathfrak t))
    \big\|^2
    \, d\mathfrak t,
\end{aligned}
\end{align}
where $0\le\mu\le1$ balances value and gradient matching,
$\|\cdot\|$ denotes the Euclidean norm in $\mathbb{R}^d$,
and $\nabla \widehat{\mathcal V}_\theta$ is computed by automatic differentiation and by the results of the previous step.
The minimization of \eqref{def:loss-fun} is performed using standard stochastic gradient methods such as ADAM \cite{kingma2014adam}.

The overall procedure is embedded within a policy iteration framework based on $\lambda$. 
At each iteration $k$, characteristic data are recomputed, and the model is retrained.
The number of trajectories $N$ controls the data coverage, while the number of optimization steps controls approximation accuracy.
In practice, the time integrals in \eqref{def:loss-fun} are replaced by discrete sums over trajectory snapshots.
To improve spatial coverage, sampling points are distributed approximately uniformly in space rather than uniformly in time;
we employ arc-length parametrization along each trajectory.
The selection of initial states
$\{\mathbf{x}_0^{(n)}\}_{n=1}^N$
influences the spatial exploration of the characteristics and should ensure adequate coverage of the state space.

\subsection{Numerical Experiments}
We present numerical tests for the quadratic control problem
$$
    f(x,a)=Ax+Ba,  \qquad \ell(x,a)=|x|^2+|a|^2, \qquad  x(0)=x_0,    
$$
with $B=\mathbb{I}_d$, $a\in \R^d$ and $A\in \R^{d\times d}$. The Hamiltonian is defined in \eqref{eq:DefnHInf}:
\begin{align}\label{eq:Htest}
    H(x,p) 
    = \frac{1}{4}|p|^2 -p\cdot Ax - |x|^2, \quad (x,p)\in \R^d\times \R^d. 
\end{align}
We choose the initial condition $u_0(x)= e^{-x^2}$, and consider the problem for $t\in[0,1]$.

\medskip

\noindent
\textit{Test configurations:}
Tables~\ref{Tab1}–\ref{Tab3} correspond to the following settings:

\begin{itemize}
\item \textbf{Table \ref{Tab1}:} $\varepsilon=0$, $h=0.01$, $d=32$. In this test, the code does not need the operator splitting step.
\item \textbf{Table \ref{Tab2}:} $\varepsilon=0.01$, $h=0.1$, $d=5$.
\item \textbf{Table \ref{Tab3}:} $\varepsilon=1$, $h=0.05$, $d=5$.
\end{itemize}

For each table, we consider two choices of $A$:
\begin{itemize}
\item \textbf{T1:} $A=\mathbb{I}_d$.
\item \textbf{T2:} $A=\frac{1}{d}(\mathfrak{a}^\top \mathfrak{a} + \mathbb{I}_d)$, where $\mathfrak{a}\in\mathbb{R}^{d\times d}$ has i.i.d. standard Gaussian entries.
\end{itemize}

\noindent 
\textit{Training procedure:}
The value function is computed in the box $[-1,1]^d$. 
Initial states $\{\mathbf{x}_0^{(n)}\}_{n=1}^N$ are sampled uniformly from this domain, 
and only trajectory data remaining inside the box are used for training. 
The loss $L(\theta)$ is minimized by full-batch ADAM \cite{kingma2014adam}.
We vary the number of characteristic trajectories $N\in\{12,\dots,20\}$, 
which controls the amount of labelled data. 
At each policy iteration, the supervised training is performed for 
$1000$ ADAM steps (or until a prescribed tolerance is reached), 
and the number of policy iterations is fixed at $30$.
Accuracy is measured by the mean residual of the stationary HJB equation,
evaluated at $N_p=10^4$ points uniformly sampled in $[-1,1]^d$:
\begin{align}\label{equ:DQ}
\mathrm{error}
=
\frac{1}{N_p}
\sum_{j=1}^{N_p}
\left|
\kappa \widehat{\mathcal{V}}_{\theta}(x^{(j)})
-
\mathcal G(x^{(j)})\cdot\nabla\widehat{\mathcal{V}}_{\theta}(x^{(j)})
-
\mathcal L(x^{(j)})
\right|.
\end{align}
For each fixed $N$, the same initial states are used for different $\mu$.
We report the average residual over the last $20$ policy iterations.

In summary, for this example, the proposed policy iteration exhibits stable and accurate convergence even when the number of characteristic trajectories or the number of training steps is limited.

\begin{table}[h!]
\centering
\small
\setlength{\tabcolsep}{5pt}
\renewcommand{\arraystretch}{1.12}

\begin{tabular}{c c c c c c c}
\hline\hline
 &  & \multicolumn{5}{c}{\textbf{The number of characteristics trajectories $N$}} \\
\cline{3-7}
 & $\mu$ & 12 & 14 & 16 & 18 & 20 \\
\hline

\multirow{6}{*}{T1}
 & 0.8 & 0.0712 & 0.1096 & 0.0387 & 0.1056 & 0.0266 \\
 & 0.6 & 0.1053 & 0.0913 & 0.0195 & 0.0131 & 0.0737 \\
 & 0.4 & 0.0881 & 0.0854 & 0.0206 & 0.0465 & 0.0655 \\
 & 0.2 & 0.0443 & 0.0998 & 0.1041 & 0.0361 & 0.0439 \\
\hline

\multirow{6}{*}{T2}
 & 0.8 & 0.0889 & 0.0761 & 0.1089 & 0.0417 & 0.1078 \\
 & 0.6 & 0.0425 & 0.0402 & 0.0352 & 0.0339 & 0.4058 \\
 & 0.4 & 0.0206 & 0.0541 & 0.0391 & 0.0255 & 0.0352 \\
 & 0.2 & 0.0628 & 0.0492 & 0.0174  & 0.0741 & 0.0203 \\
\hline

\end{tabular}

\caption{Error (HJB residual) for various $\mu$ when the number of trajectories $N$ changes. T1, T2  refer to the two tests in the text.}
\label{Tab1}
\end{table}

\begin{table}[h!]
\centering
\small
\setlength{\tabcolsep}{5pt}
\renewcommand{\arraystretch}{1.12}

\begin{tabular}{c c c c c c c}
\hline\hline
 &  & \multicolumn{5}{c}{\textbf{The number of characteristics trajectories $N$}} \\
\cline{3-7}
 & $\mu$ & 12 & 14 & 16 & 18 & 20 \\
\hline

\multirow{6}{*}{T1}
 & 0.8 & 0.0049 & 0.0054 & 0.0059 & 0.0029 & 0.0036 \\
 & 0.6 & 0.0043 & 0.0069 & 0.0041 & 0.0068 & 0.0097 \\
 & 0.4 & 0.0042 & 0.0055 & 0.0071 & 0.0033 & 0.0050 \\
 & 0.2 & 0.0086 & 0.0089 & 0.0031 & 0.0034 & 0.0072 \\
\hline

\multirow{6}{*}{T2}
 & 0.8 & 0.0044 & 0.0075 & 0.0118  & 0.0056 & 0.0038 \\
 & 0.6 & 0.0047 & 0.0034 & 0.0058  & 0.0063 & 0.0018 \\
 & 0.4 & 0.0021 & 0.0019 & 0.0051  & 0.0020 & 0.0029 \\
 & 0.2 & 0.0025 & 0.0095 & 0.0028  & 0.0057 & 0.0027 \\
\hline

\end{tabular}

\caption{Error (HJB residual) for various $\mu$ when the number of trajectories $N$ changes. T1, T2  refer to the two tests in the text.}
\label{Tab2}
\end{table}

\begin{table}[h!]
\centering
\small
\setlength{\tabcolsep}{5pt}
\renewcommand{\arraystretch}{1.12}

\begin{tabular}{c c c c c c c}
\hline\hline
 &  & \multicolumn{5}{c}{\textbf{The number of characteristics trajectories $N$}} \\
\cline{3-7}
 & $\mu$ & 12 & 14 & 16 & 18 & 20 \\
\hline

\multirow{6}{*}{T1}
 & 0.8 & 0.0118 & 0.0073 & 0.0138 & 0.0050 & 0.0167 \\
 & 0.6 & 0.0095 & 0.0075 & 0.0042 & 0.0028 & 0.0065 \\
 & 0.4 & 0.0079 & 0.0073 & 0.0040 & 0.0122 & 0.0075 \\
 & 0.2 & 0.0075 & 0.0089 & 0.0021 & 0.0047 & 0.0088 \\
\hline

\multirow{6}{*}{T2}
 & 0.8 & 0.0056 & 0.0075 & 0.0087  & 0.0039 & 0.0297 \\
 & 0.6 & 0.0055 & 0.0184 & 0.0070  & 0.0036 & 0.0167 \\
 & 0.4 & 0.0072 & 0.0018 & 0.0079  & 0.0057 & 0.0086 \\
 & 0.2 & 0.0036 & 0.0034 & 0.0053  & 0.0031 & 0.0068 \\
\hline

\end{tabular}

\caption{Error (HJB residual) for various $\mu$ when the number of trajectories $N$ changes. T1, T2  refer to the two tests in the text.}
\label{Tab3}
\end{table}

\section{Conclusion}

We propose an operator-splitting approach 
for viscous time-dependent Hamilton--Jacobi equations, 
decomposing the dynamics into a heat step and a Hamilton--Jacobi step, 
and establish quantitative error estimates with rates improving with 
the regularity of the initial data \(u_0\).
We proved exponential convergence of the 
Hamilton--Jacobi step in a weighted \(L^2\) norm, 
extending time-independent results to the time-dependent 
setting without large discount factors. The resulting PI-\(\lambda\) 
scheme is parallelizable and supports efficient 
characteristic-based (including machine learning) 
implementations for high-dimensional problems.

Open questions include obtaining optimal error estimates for the splitting scheme, as the current rates are likely suboptimal, and extending the approach to diffusion processes with jumps.

\section*{Acknowledgments}
The authors would like to thank Jiayue Han and Sheung Chi Phillip Yam for their help and discussions. 
Son Tu thanks Hung Tran and Khai Nguyen for helpful discussions on policy iteration and optimal control.

\appendix
\section{Functions of Bounded Variations} \label{Supp:BVfunctions}
Let $\Omega\subset \R^d$ be an open set.
A function $u\in L^1(\Omega)$ is said to be of \emph{bounded variation} in $\Omega$, denoted by $\mathrm{BV}(\Omega)$, if
$$
    \sup \left\lbrace \int_\Omega f\;\mathrm{div}(\phi)\;dx: \phi\in C^1_c(\Omega;\R^d), |\phi|\leq 1 \right\rbrace   < \infty. 
$$
If $u\in \mathrm{BV}(\R^d)$ then the distributional derivative $Du$ is a finite vector-valued Radon measure, 
i.e., $\partial_{x_i}u$ is a Radon measure with finite total variation 
for $i=1,\ldots, d$ (see \cite{cannarsa_semiconcave_2004, evans_measure_2015}). 
We say $u\in \mathrm{BV}_{\mathrm{loc}}(\Omega)$ if the previous property 
holds for all $\phi\in C^1_c(V;\R^d)$ for all open set $V\subset\subset \Omega$. 
For $u\in \mathrm{BV}_{\mathrm{loc}}(\Omega)$, there exists a Radon 
measure $\mu $ on $\Omega$, and a $\mu$-measurable 
function $\sigma:\Omega \to \R^d$ such that $|\sigma| = 1$ $\mu$-a.e., and 
$$
    \int_{\Omega} u \;\mathrm{div}\varphi\;dx = 
    - \int_\Omega \varphi \; d(Du) =  
    - \int_\Omega \varphi \cdot \sigma\; d\mu 
    \qquad \text{for}\;\varphi\in C^\infty_c(\Omega). 
$$
As a convention, we write $\Vert Du\Vert$ to denote the measure $\mu$ (called the \emph{variation measure} of $u$), and we write $[Df]$ to denote the vector-valued measure $\sigma \;d\mu $.  
The main property of BV functions we utilize is that, if $u\in \mathrm{BV}(\R^d)$ then 
\begin{align*}
    \int_{\R^d} |u(x+h) - u(x)|dx \leq |h|\cdot \Vert Du\Vert(\R^d). 
\end{align*}


\begin{thebibliography}{10}

\bibitem{alla_efficient_2015}
{\sc Alla, A., Falcone, M., and Kalise, D.}
\newblock An {Efficient} {Policy} {Iteration} {Algorithm} for {Dynamic} {Programming} {Equations}.
\newblock {\em SIAM Journal on Scientific Computing 37}, 1 (2015), A181--A200.
\newblock \_eprint: https://doi.org/10.1137/130932284.

\bibitem{amann_existence_1978}
{\sc Amann, H., and Crandall, M.~G.}
\newblock On {Some} {Existence} {Theorems} for {Semi}-linear {Elliptic} {Equations}.
\newblock {\em Indiana University Mathematics Journal 27}, 5 (1978), 779--790.

\bibitem{armstrong_viscosity_2015}
{\sc Armstrong, S.~N., and Tran, H.~V.}
\newblock Viscosity solutions of general viscous {Hamilton}–{Jacobi} equations.
\newblock {\em Mathematische Annalen 361}, 3 (Apr. 2015), 647--687.

\bibitem{bardi_optimal_1997}
{\sc Bardi, M., and Capuzzo-Dolcetta, I.}
\newblock {\em Optimal {Control} and {Viscosity} {Solutions} of {Hamilton}–{Jacobi}–{Bellman} {Equations}}.
\newblock Modern {Birkhäuser} {Classics}. Birkhäuser Basel, 1997.

\bibitem{talay_convergence_1997}
{\sc Barles, G.}
\newblock Convergence of {Numerical} {Schemes} for {Degenerate} {Parabolic} {Equations} {Arising} in {Finance} {Theory}.
\newblock In {\em Numerical {Methods} in {Finance}}, D.~Talay and L.~C.~G. Rogers, Eds., Publications of the {Newton} {Institute}. Cambridge University Press, Cambridge, 1997, pp.~1--21.

\bibitem{barles_convergence_1991}
{\sc Barles, G., and Souganidis, P.~E.}
\newblock Convergence of approximation schemes for fully nonlinear second order equations.
\newblock {\em Asymptotic Analysis 4}, 3 (Jan. 1991), 271--283.

\bibitem{Bea1998}
{\sc Bea, R.~W.}
\newblock Successive {G}alerkin approximation algorithms for nonlinear optimal and robust control.
\newblock {\em International Journal of Control 71}, 5 (1998), 717--743.

\bibitem{BEARD19972159}
{\sc Beard, R.~W., Saridis, G.~N., and Wen, J.~T.}
\newblock {Galerkin approximations of the generalized {Hamilton-Jacobi-Bellman} equation}.
\newblock {\em Automatica 33}, 12 (1997), 2159--2177.

\bibitem{Beard:1998ti}
{\sc Beard, R.~W., Saridis, G.~N., and Wen, J.~T.}
\newblock Approximate solutions to the time-invariant {Hamilton{\textendash}Jacobi{\textendash}Bellman} equation.
\newblock {\em Journal of Optimization Theory and Applications 96}, 3 (1998), 589--626.

\bibitem{bensoussan2018estimation}
{\sc Bensoussan, A.}
\newblock {\em {Estimation and Control of Dynamical Systems}}.
\newblock Interdisciplinary Applied Mathematics. Springer International Publishing, 2018.

\bibitem{bensoussan_value-gradient_2023}
{\sc Bensoussan, A., Han, J., Yam, S. C.~P., and Zhou, X.}
\newblock Value-{Gradient} {Based} {Formulation} of {Optimal} {Control} {Problem} and {Machine} {Learning} {Algorithm}.
\newblock {\em SIAM Journal on Numerical Analysis 61}, 2 (Apr. 2023), 973--994.
\newblock Publisher: Society for Industrial and Applied Mathematics.

\bibitem{2020handbookXZ}
{\sc {Bensoussan}, A., {Li}, Y., {Phan Cao Nguyen}, D., {Tran}, M.-B., {Yam}, S. C.~P., and {Zhou}, X.}
\newblock {Machine Learning and Control Theory}.
\newblock {\em to apper in NUMERICAL CONTROL: PART B, volume 24 of Handbook of Numerical Analysis, Elsever\/} (June 2020).

\bibitem{bensoussan_control_2019}
{\sc Bensoussan, A., and Yam, S. C.~P.}
\newblock Control problem on space of random variables and master equation.
\newblock {\em ESAIM: Control, Optimisation and Calculus of Variations 25\/} (2019), 10.
\newblock Publisher: EDP Sciences.

\bibitem{Bertsekas2001}
{\sc Bertsekas, D.~P.}
\newblock {\em Dynamic Programming and Optimal Control, Vol. I, 2nd Ed.}
\newblock Athena Scientific, Belmont, MA, 2001.

\bibitem{Bertsekas2019}
{\sc Bertsekas, D.~P.}
\newblock {\em Reinforcement Learning and Optimal Control}.
\newblock Athena Scientific, Belmont, MA, 2019.

\bibitem{calder2024}
{\sc Calder, J.}
\newblock Lecture notes on viscosity solutions, July 2024.

\bibitem{camilli_quantitative_2024}
{\sc Camilli, F., Goffi, A., and Mendico, C.}
\newblock Quantitative and qualitative properties for {Hamilton}-{Jacobi} {PDEs} via the nonlinear adjoint method, Nov. 2024.
\newblock arXiv:2307.12932 [math].

\bibitem{cannarsa_semiconcave_2004}
{\sc Cannarsa, P., and Sinestrari, C.}
\newblock {\em Semiconcave {Functions}, {Hamilton}—{Jacobi} {Equations}, and {Optimal} {Control}}.
\newblock Birkhäuser, Boston, MA, 2004.

\bibitem{chaintron_optimal_2025}
{\sc Chaintron, L.-P., and Daudin, S.}
\newblock Optimal rate of convergence in the vanishing viscosity for quadratic {Hamilton}-{Jacobi} equations, Feb. 2025.
\newblock arXiv:2502.09103 [math].

\bibitem{ChowOsher2017JSC}
{\sc Chow, Y.~T., Darbon, J., Osher, S., and Yin, W.}
\newblock Algorithm for overcoming the curse of dimensionality for time-dependent non-convex hamilton--jacobi equations arising from optimal control and differential games problems.
\newblock {\em Journal of Scientific Computing 73}, 2 (2017), 617--643.

\bibitem{ChowOsherADMM2018}
{\sc Chow, Y.~T., Darbon, J., Osher, S., and Yin, W.}
\newblock Algorithm for overcoming the curse of dimensionality for certain non-convex {Hamilton--Jacobi} equations, projections and differential games.
\newblock {\em Annals of Mathematical Sciences and Applications 3}, 2 (2018), 369--403.

\bibitem{ChowOsher2019JSC}
{\sc Chow, Y.~T., Li, W., Osher, S., and Yin, W.}
\newblock Algorithm for {Hamilton--Jacobi} equations in density space via a generalized {Hopf} formula.
\newblock {\em Journal of Scientific Computing 80}, 2 (2019), 1195--1239.

\bibitem{cirant_convergence_2025}
{\sc Cirant, M., and Goffi, A.}
\newblock Convergence rates for the vanishing viscosity approximation of {Hamilton}-{Jacobi} equations: the convex case, Feb. 2025.
\newblock arXiv:2502.15495 [math].

\bibitem{crandall_users_1992}
{\sc Crandall, M.~G., Ishii, H., and Lions, P.-L.}
\newblock User's guide to viscosity solutions of second order partial differential equations.
\newblock {\em Bull. Amer. Math. Soc. (N.S.) 27 (1992) 1-67\/} (July 1992).

\bibitem{crandall_two_1984}
{\sc Crandall, M.~G., and Lions, P.~L.}
\newblock Two {Approximations} of {Solutions} of {Hamilton}-{Jacobi} {Equations}.
\newblock {\em Mathematics of Computation 43}, 167 (1984), 1--19.
\newblock Publisher: American Mathematical Society.

\bibitem{DarbonOsher2016RMS}
{\sc Darbon, J., and Osher, S.}
\newblock Algorithms for overcoming the curse of dimensionality for certain {Hamilton--Jacobi} equations arising in control theory and elsewhere.
\newblock {\em Research in the Mathematical Sciences 3}, 1 (2016), 1--26.

\bibitem{dutta_rate_2026}
{\sc Dutta, P., Nguyen, K.~T., and Tu, S. N.~T.}
\newblock On the rate of convergence in superquadratic hamilton–jacobi equations with state constraints.
\newblock {\em Nonlinear Differential Equations and Applications {NoDEA} 33}, 3 (2026), 63.

\bibitem{2020HighPDEReivewE}
{\sc E, W., Han, J., and Jentzen, A.}
\newblock Algorithms for solving high dimensional {PDEs}: From nonlinear {M}onte {C}arlo to machine learning.
\newblock {\em arXiv preprint arXiv:2008.13333\/} (2020).

\bibitem{evans_adjoint_2010}
{\sc Evans, L.~C.}
\newblock Adjoint and {Compensated} {Compactness} {Methods} for {Hamilton}–{Jacobi} {PDE}.
\newblock {\em Archive for Rational Mechanics and Analysis 197}, 3 (Sept. 2010), 1053--1088.

\bibitem{evans_measure_2015}
{\sc Evans, L.~C., and Gariepy, R.~F.}
\newblock {\em Measure {Theory} and {Fine} {Properties} of {Functions}, {Revised} {Edition}}.
\newblock Chapman and Hall/CRC, New York, Apr. 2015.

\bibitem{falcone2013semi}
{\sc Falcone, M., and Ferretti, R.}
\newblock {\em Semi-Lagrangian approximation schemes for linear and {Hamilton—Jacobi} equations}.
\newblock SIAM, 2013.

\bibitem{fathi_weak_2008}
{\sc Fathi, A.}
\newblock {\em The {Weak} {KAM} {Theorem} in {Lagrangian} {Dynamics}}.
\newblock Cambridge University Press, Mar. 2008.

\bibitem{fleming2006controlled}
{\sc Fleming, W., and Soner, H.}
\newblock {\em Controlled Markov Processes and Viscosity Solutions}.
\newblock Stochastic Modelling and Applied Probability. Springer New York, 2006.

\bibitem{fleming_convergence_1961}
{\sc Fleming, W.~H.}
\newblock The convergence problem for differential games.
\newblock {\em Journal of Mathematical Analysis and Applications 3}, 1 (1961), 102--116.

\bibitem{Fleming1975Book}
{\sc Fleming, W.~H., and Rishel, R.~W.}
\newblock {\em {Deterministic and Stochastic Optimal Control}}.
\newblock Stochastic Modelling and Applied Probability. Springer New York, 1975.

\bibitem{guo_policy_2025}
{\sc Guo, X., Tran, H.~V., and Zhang, Y.~P.}
\newblock Policy iteration for nonconvex viscous {Hamilton}--{Jacobi} equations, Mar. 2025.
\newblock arXiv:2503.02159 [math].

\bibitem{han_remarks_2022}
{\sc Han, Y., and Tu, S. N.~T.}
\newblock Remarks on the {Vanishing} {Viscosity} {Process} of {State}-{Constraint} {Hamilton}–{Jacobi} {Equations}.
\newblock {\em Applied Mathematics \& Optimization 86}, 1 (June 2022), 3.

\bibitem{7040310}
{\sc {Horowitz}, M.~B., {Damle}, A., and {Burdick}, J.~W.}
\newblock Linear {Hamilton-Jacobi-Bellman} equations in high dimensions.
\newblock In {\em 53rd IEEE Conference on Decision and Control\/} (2014), pp.~5880--5887.

\bibitem{Izzo01145}
{\sc Izzo, D., \"{O}zt\"{u}rk, E., and M\"{a}rtens, M.}
\newblock Interplanetary transfers via deep representations of the optimal policy and/or of the value function.
\newblock In {\em Proceedings of the Genetic and Evolutionary Computation Conference Companion\/} (New York, NY, USA, 2019), GECCO '19, Association for Computing Machinery, p.~1971–1979.

\bibitem{jakobsen_convergence_2001}
{\sc Jakobsen, E.~R., Karlsen, K.~H., and Risebro, N.~H.}
\newblock On the {Convergence} {Rate} of {Operator} {Splitting} for {Hamilton}--{Jacobi} {Equations} with {Source} {Terms}.
\newblock {\em SIAM Journal on Numerical Analysis 39}, 2 (Jan. 2001), 499--518.

\bibitem{KK2018}
{\sc Kalise, D., and Kunisch, K.}
\newblock Polynomial approximation of high-dimensional {Hamilton-Jacobi-Bellman} equations and applications to feedback control of semilinear parabolic {PDES}.
\newblock {\em SIAM Journal on Scientific Computing 40}, 2 (2018), A629--A652.

\bibitem{Kang2015causality}
{\sc Kang, W., and Wilcox, L.}
\newblock A causality free computational method for {HJB} equations with application to rigid body satellites.
\newblock In {\em AIAA Guidance, Navigation, and Control Conference\/} (2015), p.~2009.

\bibitem{Kang2017}
{\sc Kang, W., and Wilcox, L.~C.}
\newblock Mitigating the curse of dimensionality: sparse grid characteristics method for optimal feedback control and {HJB} equations.
\newblock {\em Computational Optimization and Applications 68}, 2 (2017), 289--315.

\bibitem{kim_physics-informed_2025}
{\sc Kim, Y., Cho, N., Kim, M., and Kim, Y.}
\newblock Physics-informed approach for exploratory hamilton--jacobi--bellman equations via policy iterations.

\bibitem{kim_neural_2025}
{\sc Kim, Y., Kim, Y., Kim, M., and Cho, N.}
\newblock Neural policy iteration for stochastic optimal control: A physics-informed approach.

\bibitem{kingma2014adam}
{\sc Kingma, D.~P., and Ba, J.}
\newblock Adam: A method for stochastic optimization.
\newblock {\em arXiv preprint arXiv:1412.6980\/} (2014).

\bibitem{langseth_convergence_1996}
{\sc Langseth, J.~O., Tveito, A., and Winther, R.}
\newblock On the convergence of operator splitting applied to conservation laws with source terms.
\newblock {\em {SIAM} Journal on Numerical Analysis 33}, 3 (1996), 843--863.

\bibitem{le_dynamical_2017}
{\sc Le, N.~Q., Mitake, H., and Tran, H.~V.}
\newblock {\em Dynamical and Geometric Aspects of Hamilton-Jacobi and Linearized Monge-Ampère Equations}, vol.~2183 of {\em Lecture Notes in Mathematics}.
\newblock Springer International Publishing.

\bibitem{lee_hamiltonjacobi_2025}
{\sc Lee, J.~Y., and Kim, Y.}
\newblock Hamilton–jacobi based policy-iteration via deep operator learning.
\newblock {\em Neurocomputing 646\/} (2025), 130515.

\bibitem{ChowOshersplitting2018}
{\sc Lin, A.~T., Chow, Y.~T., and Osher, S.~J.}
\newblock A splitting method for overcoming the curse of dimensionality in {Hamilton–Jacobi} equations arising from nonlinear optimal control and differential games with applications to trajectory generation.
\newblock {\em Communications in Mathematical Sciences 16}, 7 (1 2018).

\bibitem{lions_generalized_1982}
{\sc Lions, P.-L.}
\newblock {\em Generalized {Solutions} of {Hamilton}–{Jacobi} {Equations}}.
\newblock Chapman \& {Hall}/{CRC} research notes in mathematics series. Pitman, 1982.

\bibitem{luca_gerardo-giorda_parallelizing_2014}
{\sc {Luca Gerardo-Giorda}, and Tran, M.~B.}
\newblock Parallelizing the {Kolmogorov}-{Fokker}-{Planck} equation.
\newblock {\em ESAIM: Mathematical Modelling and Numerical Analysis\/} (Sept. 2014).

\bibitem{Kang2019}
{\sc Nakamura-Zimmerer, T., Gong, Q., and Kang, W.}
\newblock Adaptive deep learning for high-dimensional hamilton--jacobi--bellman equations.
\newblock {\em SIAM Journal on Scientific Computing 43}, 2 (2021), A1221--A1247.

\bibitem{Kang2021QRnet}
{\sc {Nakamura-Zimmerer}, T., {Gong}, Q., and {Kang}, W.}
\newblock Qrnet: Optimal regulator design with lqr-augmented neural networks.
\newblock {\em IEEE Control Systems Letters 5}, 4 (2021), 1303--1308.

\bibitem{OSHER198812}
{\sc Osher, S., and Sethian, J.~A.}
\newblock Fronts propagating with curvature-dependent speed: Algorithms based on {Hamilton-Jacobi} formulations.
\newblock {\em Journal of Computational Physics 79}, 1 (1988), 12--49.

\bibitem{Oster2019}
{\sc Oster, M., Sallandt, L., and Schneider, R.}
\newblock Approximating the stationary {Hamilton-Jacobi-Bellman} equation by hierarchical tensor products.
\newblock {\em arXiv: 1911.00279\/} (nov 2019).

\bibitem{puterman_convergence_1979}
{\sc Puterman, M.~L., and Brumelle, S.~L.}
\newblock On the {Convergence} of {Policy} {Iteration} in {Stationary} {Dynamic} {Programming}.
\newblock {\em Mathematics of Operations Research 4}, 1 (1979), 60--69.

\bibitem{qian_optimal_2024}
{\sc Qian, J., Sprekeler, T., Tran, H.~V., and Yu, Y.}
\newblock Optimal {Rate} of {Convergence} in {Periodic} {Homogenization} of {Viscous} {Hamilton}-{Jacobi} {Equations}.
\newblock {\em Multiscale Modeling \& Simulation 22}, 4 (Dec. 2024), 1558--1584.
\newblock Publisher: Society for Industrial and Applied Mathematics.

\bibitem{Recht2019tour}
{\sc Recht, B.}
\newblock {A Tour of Reinforcement Learning: The View from Continuous Control}.
\newblock {\em Annual Review of Control, Robotics, and Autonomous Systems 2}, 1 (2019), 253--279.

\bibitem{souganidis_max-min_1985}
{\sc Souganidis, P.~E.}
\newblock Max-min representations and product formulas for the viscosity solutions of {Hamilton}-{Jacobi} equations with applications to differential games.
\newblock {\em Nonlinear Analysis: Theory, Methods \& Applications 9}, 3 (Mar. 1985), 217--257.

\bibitem{sun_alternating_1996}
{\sc Sun, M.}
\newblock Alternating direction algorithms for solving {Hamilton}-{Jacobi}-{Bellman} equations.
\newblock {\em Applied Mathematics and Optimization 34}, 3 (Nov. 1996), 267--277.

\bibitem{sutton2018reinforcement}
{\sc Sutton, R.~S., and Barto, A.~G.}
\newblock {\em {Reinforcement Learning: An Introduction}}.
\newblock Adaptive Computation and Machine Learning series. MIT Press, 2018.

\bibitem{tang_policy_2025}
{\sc Tang, W., Tran, H.~V., and Zhang, Y.~P.}
\newblock Policy {Iteration} for the {Deterministic} {Control} {Problems}—{A} {Viscosity} {Approach}.
\newblock {\em SIAM Journal on Control and Optimization 63}, 1 (Feb. 2025), 375--401.
\newblock Publisher: Society for Industrial and Applied Mathematics.

\bibitem{tran_adjoint_2011}
{\sc Tran, H.~V.}
\newblock Adjoint methods for static {Hamilton}–{Jacobi} equations.
\newblock {\em Calculus of Variations and Partial Differential Equations 41}, 3 (July 2011), 301--319.

\bibitem{tran_hamilton-jacobi_2021}
{\sc Tran, H.~V.}
\newblock {\em Hamilton-{Jacobi} {Equations}: {Theory} and {Applications}}, vol.~213 of {\em Graduate {Studies} in {Mathematics}}.
\newblock American Mathematical Society, 2021.

\bibitem{FastSweep2003}
{\sc Tsai, Y.-H.~R., Cheng, L.-T., Osher, S., and Zhao, H.-K.}
\newblock Fast sweeping algorithms for a class of {Hamilton--Jacobi} equations.
\newblock {\em SIAM Journal on Numerical Analysis 41}, 2 (2003), 673--694.

\bibitem{FastMarching1994}
{\sc {Tsitsiklis}, J.~N.}
\newblock Efficient algorithms for globally optimal trajectories.
\newblock In {\em Proceedings of 1994 33rd IEEE Conference on Decision and Control\/} (1994), vol.~2, pp.~1368--1373 vol.2.

\bibitem{wang_vanishing_2025}
{\sc Wang, Z., and Zhang, J.}
\newblock On the vanishing viscosity limit of {Hamilton}-{Jacobi} equations with nearly optimal discount, Sept. 2025.
\newblock arXiv:2509.17402 [math].

\bibitem{yang_solving_2025}
{\sc Yang, H.~J., Gim, M., and Kim, Y.}
\newblock Solving nonconvex {Hamilton--Jacobi--Isaacs} equations with {PINN}-based policy iteration.

\end{thebibliography}
\end{document}